\documentclass[sn-mathphys-num,pdflatex]{sn-jnl}% Math and Physical Sciences Numbered Reference Style 

\usepackage{graphicx}%
\usepackage{multirow}%
\usepackage{amsmath,amssymb,amsfonts}%
\usepackage{amsthm}%
\usepackage{mathrsfs}%
\usepackage[title]{appendix}%
\usepackage{xcolor}%
\usepackage{textcomp}%
\usepackage{manyfoot}%
\usepackage{booktabs}%
\usepackage{algorithm}%
\usepackage{algorithmicx}%
\usepackage{algpseudocode}%
\usepackage{listings}%
\usepackage{soul}%

\usepackage{bm,bbm}
\usepackage{enumerate}
\usepackage[utf8]{inputenc}
\usepackage{mparhack}
\usepackage{marginnote}
\usepackage{hyperref}
\usepackage{anyfontsize}

\usepackage{tikz}           % Drawing package
\usepackage{tikz-cd}
\usepackage{pstricks,pst-node,pst-text,pst-3d,graphics, bm}
\usepackage{pgfplots}
\pgfplotsset{compat=1.18}

\newtheorem{theorem}{Theorem}%  meant for continuous numbers
\newtheorem{remark}[theorem]{Remark}
\newtheorem{assumption}{Assumption}
\newtheorem{definition}{Definition}%

\DeclareMathOperator{\curl}{curl}

\DeclareMathOperator{\rank}{rank}
\newcommand{\FEM}{\scalebox{.7}{\rm FE}}
\newcommand{\ODE}{\scalebox{.7}{\rm ODE}}
\newcommand{\POD}{\scalebox{.7}{\rm POD}}
\newcommand{\DEIM}{\scalebox{.7}{\rm DEIM}}
\newcommand{\sPOD}{\scalebox{.5}{\rm POD}}
\newcommand{\sDEIM}{\scalebox{.5}{\rm DEIM}}

\DeclareMathAlphabet{\mathpzc}{OT1}{pzc}{m}{it}
\newcommand{\cEl}{\scalebox{1.3}{$\mathpzc{E}$}}

\newcommand{\cBl}{\scalebox{1.3}{$\mathpzc{B}$}}
\newcommand{\cCl}{\scalebox{1.3}{$\mathpzc{C}$}}
\newcommand{\cAl}{\scalebox{1.3}{$\mathpzc{A}$}}

\begin{document}

\title[Regularization and passivity-preserving model reduction of quasilinear MQS problems]{Regularization and passivity-preserving model reduction of quasilinear magneto-quasistatic coupled problems}

\author[1]{\fnm{Johanna} \sur{Kerler-Back}}

\author[2]{\fnm{Timo} \sur{Reis}}

\author*[3]{\fnm{Tatjana} \sur{Stykel}}\email{tatjana.stykel@uni-a.de}

\affil[1]{
\orgdiv{Institut f\"ur Mathematik}, 
\orgname{Universit\"at Augsburg}, 
\orgaddress{\street{Universit\"atsstra{\ss}e~12a}, 
\postcode{86159}, 
\city{Augsburg}, 
\country{Germany}}}

\affil[2]{
\orgdiv{Institut f\"ur Mathematik}, 
\orgname{Technische Universit\"at Ilmenau}, \\
\orgaddress{\street{Weimarer Stra{\ss}e~25}, 
\postcode{98693}, 
\city{Ilmenau}, 
\country{Germany}}}

\affil*[3]{
\orgdiv{Institut f\"ur Mathematik \& Centre for Advanced Analytics and Predictive Sciences (CAAPS)}, \orgname{Universit\"at Augsburg}, 
\orgaddress{\street{Universit\"atsstra{\ss}e~12a}, 
\postcode{86159}, 
\city{Augsburg}, 
\country{Germany}}}

\abstract{
We consider the quasilinear magneto-quasistatic field equations that arise in the simulation of low-frequency electromagnetic devices coupled to electrical circuits. Spatial discretization of these equations on 3D~domains using the finite element method results in a singular system of differential-algebraic equations (DAEs). First, we analyze the structural properties of this system and present a novel regularization approach based on projecting out the singular state components. Next, we explore the passivity of the variational magneto-quasistatic problem and its discretization by defining suitable storage functions. For model reduction of the magneto-quasistatic system, we employ the proper orthogonal decomposition (POD) technique combined with the discrete empirical interpolation method (DEIM), to facilitate efficient evaluation of the system's nonlinearities. Our model reduction approach involves the transformation of the regularized DAE into a system of ordinary differential equations, leveraging a special block structure inherent in the problem, followed by applying standard model reduction techniques to the transformed system.
We prove that the POD-reduced model preserves passivity, and for the POD-DEIM-reduced model, we propose to enforce passivity by perturbing the output in a way that accounts for DEIM errors. Numerical experiments illustrate the effectiveness of the presented model reduction methods and the passivity enforcement technique.
}



\maketitle

\section{Introduction} 

Maxwell's equations describe the dynamic behavior of electromagnetic
systems by relating the electric and magnetic fields in a~medium. Together with material-dependent constitutive relations for field intensities and flux densities, they form the basis for all electromagnetic phenomena \cite{Boss98,Grif17,Jack99}.

Assuming that the contribution of the displacement currents is negligible compared to the conductive currents, the magnetic field can be described by magneto-quasistatic (MQS) equations, which can be considered as an~approximation to Maxwell's equations. Such an~assumption is justified, for example, for electromagnetic devices operating at low frequencies. Due to the presence of electrically conducting and non-conducting spatial subdomains, the MQS equations become of mixed parabolic-elliptic type. Moreover, if the relationship between the magnetic field and the flux intensities is nonlinear, and the inductive coupling of the electromagnetic components to an~external electrical circuit is included, the MQS system amounts to quasilinear partial integro-differential-algebraic equations whose dynamics is restricted to a~manifold described by algebraic and integral equations. 

A~comprehensive analysis of the MQS equations, ranging from existence, uniqueness and regularity of solutions to the passivity and stability properties, has been presented in \cite{ChiRS23, ReiS23}. The structural properties of coupled field/circuit problems have been studied in \cite{BartBS11, CGdGS20}. For the numerical simulation of such problems, spatial discretization methods such as the finite integration technique (FIT) or the finite element method (FEM) combined with appropriate time integration schemes are commonly used in practice; see, e.g., \cite{AloV10,Boss01,CleW01,Schoeps11} and references therein.

Discretizing the MQS systems on 3D domains leads to differential-algebraic equations (DAEs), which pose significant analytical and computational challenges
due to inherent singularity and high dimensionality.  
For distributed MQS systems and their spatial discretizations, several regularization approaches have been presented in the literature. Most of them are based on different gauging techniques such as tree-cotree gauging \cite{ManC95},
multilevel gauging \cite{Hip00},
grad-div gauging \cite{AloV10,Boss01,CleW02}, 
and ghost field gauging \cite{SchMM02}. In this paper, we take a~different approach
and present, as a first contribution, a~novel regularization strategy which is based on the computation of a~condensed form for the quasilinear coupled MQS system by using a~linear coordinate transformation. This transformation allows us not only to identify underdetermined state components and redundant equations, but also to transform the DAE system into one governed by an~ordinary differential equation (ODE) while preserving passivity. 

In order to reduce the computational effort in transient calculations, reduced-order modelling can be employed. Over the last three decades, many different model reduction techniques have been developed and successfully applied to various types of physical and engineering problems, see \cite{MOR21} for an extensive collection of algorithms and applications.
Model reduction of linear MQS systems using balanced truncation has been considered in \cite{KerBS17,KerS22}, while model reduction methods for nonlinear MQS problems based on the proper orthogonal decomposition (POD) technique combined with the discrete empirical interpolation method (DEIM) have been presented in \cite{KerBS17,MonPHC17,SatCI16}. Most existing works on model reduction demonstrate the approximation properties of the resulting reduced-order models and report simulation time savings, but they often overlook the preservation of the underlying physical properties. To the best of our knowledge, passivity preservation for MQS systems has only been addressed for linear problems \cite{KerS22}. Our second contribution is passivity analysis for spatially discretized and POD-(DEIM-)reduced quasilinear MQS systems.  
We show that the POD method preserves the underlying symmetric structure of the MQS system and, as a~consequence, ensures the preservation of passivity.  
Unfortunately, when the nonlinearity is further approximated by DEIM, its symmetric structure is lost and the preservation of passivity can no longer be guaranteed. 
To overcome this difficulty, we develop a~passivity enforcement method which is based on the derivation of a~computable error bound on the DEIM state error and perturbation of the output components. We will mainly focus on the more general 3D case and only comment on the 2D case where differences arise.

The paper is organized as follows. In Section~\ref{sec:model}, we introduce a quasilinear coupled MQS model, collect the assumptions on a~spatial domain and material parameters, and review some results from \cite{ChiRS23, ReiS23} on unique solvability and passivity properties of this model. A~FEM discretization of the MQS system, a~projection-based regularization and passivity of the FEM model are discussed in Section~\ref{sec:fem}. Section~\ref{sec:pod} deals with POD model reduction of the regularized MQS model. We prove that this method preserves passivity in the reduced model. In Section~\ref{sec:deim}, we develop a~perturbation-based approach to enforce passivity in the POD-DEIM-reduced model. Some results of numerical experiments are given in Section~\ref{sec:num}. Finally, Section~\ref{sec:concl} contains concluding remarks.

{\bf Notations.} Throughout this paper, all spaces are real. A~set of all nonnegative real numbers is denoted by $\mathbb{R}_{\geq 0}$, and $\mathbb{R}^{n\times m}$ stands for the set of real $n\times m$ matrices. The image and the kernel of a~matrix~$A$ are denoted by $\mbox{im}(A)$ and $\mbox{ker}(A)$, respectively. For elements of $\mathbb{R}^n$ and $\mathbb{R}^{m\times n}$, $\|\cdot\|$ stands for the Euclidean vector norm and the spectral matrix norm defined by the largest singular value, respectively. We denote by $\nabla \varphi$, $\nabla\times \varphi$, and $\nabla\cdot\varphi$ the weak gradient, curl, and divergence, respectively. For a~bounded Lipschitz domain $\varOmega\subset\mathbb{R}^3$, let $L^2(\varOmega;\mathbb{R}^3)$ denote the Lebesgue space of square integrable functions with values in $\mathbb{R}^3$. This space is equipped with the standard inner product $\langle\cdot,\cdot\rangle_{L^2(\varOmega;\mathbb{R}^3)}$. 
In addition, $H^1_0(\varOmega)$ is the space of square integrable functions with weak gradient and vanishing boundary trace, and $H_0(\mathrm{curl},\varOmega)$ is the space of square integrable functions with weak curl and vanishing tangential component of the boundary trace. A~dual space of a~Hilbert space $\mathcal{X}$ is denoted by~$\mathcal{X}'$, and \mbox{$\langle\cdot,\cdot\rangle:\mathcal{X}\times \mathcal{X}'\to\mathbb{R}$} stands for a~canonical dual pairing.

\section{Model problem}
\label{sec:model}

Let $\varOmega\subset\mathbb{R}^3$ be a~bounded simply connected domain with a~Lipschitz boundary $\partial\varOmega$ and let $T>0$. We consider the quasilinear MQS system  
\begin{subequations}\label{eq:MQS}
	\begin{align}
		{\tfrac{\partial}{\partial t}}\left(\sigma\bm{A}\right) + \nabla \times \left(\nu(\cdot,\|\nabla \times \bm{A}\|)
		\nabla \times \bm{A}\right)  = & \; \chi\,\bm{i} & \text{ in } &\varOmega\times (0,T], \label{eq:MQS1} \\
		\tfrac{{\rm d}}{{\rm d} t} \int_\varOmega \chi^T\bm{A} \, {\rm d}\xi + R\, \bm{i} = &\;\bm{v}\label{eq:MQScoupl} & \text{ on } &(0,T], \\
		\bm{A}\times \bm{n}_o  = &\; 0 & \mbox{in }& \partial \varOmega\times (0,T],
		\label{eq:MQSbc}\\[2mm]
		\sigma\bm{A}(\cdot,0) = &\; \sigma\bm{A}_0 &\text{ in }&\varOmega,
		\label{eq:MQSic1}\\
		\int_\varOmega \chi^T\bm{A}(\cdot,0) \, {\rm d}\xi=&\, \int_\varOmega \chi^T\bm{A}_0\,
		{\rm d}\xi,&& \label{eq:MQSic2}
	\end{align}
\end{subequations}
where $\bm{A}:\overline{\varOmega}\times [0,T]\to\mathbb{R}^3$ is the magnetic vector potential, 
$\nu:\varOmega\times\mathbb{R}_{\geq 0}\to\mathbb{R}_{\geq 0}$ is the magnetic reluctivity, 
$\sigma:\varOmega\to\mathbb{R}_{\geq 0}$ is the electric conductivity, 
\mbox{$\chi:\varOmega\to\mathbb{R}^{3\times m}$} the~winding density function  describing the geometry of $m$ windings,  
$\bm{v}:[0,T]\to\mathbb{R}^m$ is the voltage, and 
\mbox{$\bm{i}:[0,T]\to\mathbb{R}^m$} is the electrical current through the electromagnetic conductive contacts. Equation \eqref{eq:MQS1} describing the dynamics of the magnetic vector potential $\bm{A}$ results from Maxwell's equations by neglecting the displacement currents and exploiting the constitutive relations. Further, equation~\eqref{eq:MQScoupl} with the resistance matrix $R\in\mathbb{R}^{m\times m}$ follows from Faraday's law of induction. It describes the coupling of electromagnetic devices to an~external circuit. 
The boundary condition \eqref{eq:MQSbc} with the outer unit normal vector 
\mbox{$\bm{n}_o:\partial \varOmega\to\mathbb{R}^3$} implies that the magnetic flux through the boundary $\partial\Omega$ vanishes. 
Finally, equations \eqref{eq:MQSic1} and \eqref{eq:MQSic2} with given \mbox{$\bm{A}_0:\varOmega\to\mathbb{R}^3$} provide the initial conditions for the magnetic vector potential. 
The coupling of electromagnetic devices to an~external electric network is realized here as 
a~stranded conductor model with $m$ ports, where stranded conductors behave as current-driven circuit elements, see \cite{SchoepsDGW13} for details. The coupling interface is described by the coupling equation \eqref{eq:MQScoupl} enhanced with the output equation
\begin{equation}
y=\bm{i}.
\label{eq:output}
\end{equation}
The resulting system \eqref{eq:MQS}, \eqref{eq:output} becomes the control system with the input $\bm{v}$, the state~$(\bm{A},\bm{i})$ and the output $\bm{i}$.

\begin{remark}
In the 2D case, the magnetic vector potential and the winding function have the form 
$\bm{A}=[0,0,\bm{A}_3]^T$ and $\chi=[0,0,\chi_3^T]^T$, respectively. In this case, equations 
\eqref{eq:MQS1} and \eqref{eq:MQSbc} are reduced to 
\[
\arraycolsep=2pt
\begin{array}{rclrl}
\displaystyle{\tfrac{\partial}{\partial t}(\sigma\bm{A}_3) - \nabla \cdot \left(\nu(\cdot,\|\nabla \bm{A}_3\|) 
\nabla \bm{A}_3\right)} & = & \chi_3\,\bm{i} & \quad \mbox{in}& \enskip\varOmega\times (0,T], \\[1mm]
\bm{A}_3 & = & 0 & \quad \mbox{on}& \partial \varOmega\times (0,T]. \\[1mm]
\end{array} 
\]
The coupling equation \eqref{eq:MQScoupl} and the initial conditions \eqref{eq:MQSic1}, \eqref{eq:MQSic2} can be simplified accor\-dingly.
\end{remark}

Next, we collect assumptions on the spatial domain, material parameters and winding function.

\begin{assumption}[Spatial domain, material parameters, winding function] 
\label{ass:assumptions}\
{\rm
\begin{enumerate}[{\rm a)}]
\item \label{ass:domain}
    Let $\varOmega\subset\mathbb{R}^3$ be a~simply connected bounded Lipschitz domain which is decomposed into the {\em conducting} and {\em non-conducting subdomains} $\varOmega_{C}$ and $\varOmega_{I}$, respectively, such that $\overline{\varOmega}_{C}\subset \varOmega$ and \mbox{$\varOmega_I=\varOmega\setminus \overline{\varOmega}_{C}$}. Furthermore, $\varOmega_{C}$ is connected, and $\varOmega_{I}$ has finitely many connected internal subdomains $\varOmega_{I,1},\ldots,\varOmega_{I,q}$ with single boundary components $\mathit{\Gamma}_1,\ldots,\mathit{\Gamma}_q$, respectively, and an~external subdomain $\varOmega_{I,{\rm ext}}$ which has two boundary components $\mathit{\Gamma}_{{\rm ext}}=\overline{\varOmega}_{I,{\rm ext}}\cap \overline{\varOmega}_{C}$ and $\partial\varOmega$. 

\item \label{ass:conductivity}
    The {\em electric conductivity} is given by $\sigma(\xi)=\sigma_C \mathbbm{1}_{\varOmega_C}(\xi)$, where  \mbox{$\mathbbm{1}_{\varOmega_C}:\varOmega\to\{0,1\}$} is an~indicator function of the subdomain $\varOmega_C$ and $\sigma_C>0$. 

\item \label{ass:reluctivity}
    The {\em magnetic reluctivity} is given by $\nu(\xi,\zeta)= \nu_C(\zeta)\mathbbm{1}_{\varOmega_C}(\xi) + \nu_I\mathbbm{1}_{\varOmega_I}(\xi)$, where $\nu_I >0$ and $\nu_C:\mathbb{R}_{\geq 0}\to\mathbb{R}_{\geq 0}$ satisfies the following conditions:
\begin{enumerate}[{\rm (i)}]
    \item $\nu_C$ is measurable; 
    \item $\zeta\mapsto\nu_C(\zeta)\zeta$ is strongly monotone with a~monotonicity constant $m_{\nu_C}>0$;
    \item $\zeta\mapsto\nu_C(\zeta)\zeta$ is Lipschitz continuous with a~Lipschitz constant $L_{\nu_C}>0$.
    \end{enumerate}

\item \label{ass:resistance}
    The {\em resistance matrix} $R\in\mathbb{R}^{m\times m}$ is symmetric and positive definite. 

\item\label{ass:winding}
    The columns of the {\em winding function} $\chi\!=\![\chi_1,\ldots,\chi_m]$ have the following properties:
\begin{enumerate}[{\rm (i)}]
 \item $\chi_j\in L^2(\varOmega;\mathbb{R}^3)$ with $\nabla\cdot\chi_j=0$ for $j=1,\ldots, m$;
\item $\mathrm{supp}(\chi_i)\cap\mathrm{supp}(\chi_j)=\emptyset$ for $i,j=1,\ldots,m$ and $i\neq j$.
\end{enumerate}
\end{enumerate}
}
\end{assumption}

\begin{remark}\label{rem:reluct_wind}\
\begin{enumerate}[{\rm a)}]
\item \label{rem:reluct}
Obviously, $\nu$ inherits the properties of $\nu_C$, i.e.,
for all \mbox{$\xi\in\varOmega$}, \mbox{$\zeta\mapsto\nu(\xi,\zeta)\zeta$} is strongly monotone with the monotonicity constant \mbox{$m_\nu=\min(m_{\nu_C},\nu_I)$} and 
Lipschitz continuous with the Lipschitz constant \mbox{$L_\nu=\max(L_{\nu_C},\nu_I)$}.  

\item \label{rem:wind}
The divergence-free condition for $\chi_1,\ldots,\chi_m$ implies that there exists a~matrix-valued function 
\mbox{$\gamma=[\gamma_1,\ldots,\gamma_m]:\varOmega\to\mathbb{R}^{3\times m}$} with 
components in $H^1(\varOmega)$ such that $\chi_j=\nabla\times \gamma_j$ for $j=1,\ldots,m$, 
see \textup{\cite[Thm.~3.4]{GiraRavi86}}.
\end{enumerate}
\end{remark}

\subsection{Weak formulation} \label{sec:weaksol}

We now present a~weak formulation for the MQS problem \eqref{eq:MQS} and briefly review some results from \cite{ChiRS23} on existence, uniqueness and regularity of a weak solution of this problem. 

We start with introducing appropriate function spaces. Let $\varOmega\subset\mathbb{R}^3$ with the subdomains $\varOmega_C,\varOmega_I\subset\varOmega$ be as in Assumption~\ref{ass:assumptions}~\ref{ass:domain}).
Let $X(\varOmega,\varOmega_C)$ denote the space of square integrable functions which are $L^2$-orthogonal to all gradient fields of functions from $H^1_0(\varOmega)$ being constant on each interface component $\mathit{\Gamma}_1,\ldots,\mathit{\Gamma}_q$ and $\mathit{\Gamma}_{\rm ext}$. Note that $X(\varOmega,\varOmega_C)$ is a~Hilbert space equipped with the inner product in $L^2(\varOmega;\mathbb{R}^3)$. We also define the space $X_0(\mathrm{curl},\varOmega,\varOmega_C)=H_0(\mathrm{curl},\varOmega)\cap X(\varOmega,\varOmega_C)$ which is a~Hilbert space equipped with the inner product in $H_0(\mathrm{curl},\varOmega)$. 

Multiplying equations \eqref{eq:MQS1} and \eqref{eq:MQSic1} with a~test function 
$\varphi\in X_0(\mbox{curl},\varOmega, \varOmega_C)$ 
and integrating them over the domain $\varOmega$, we obtain by using the integration by parts formula
\cite[Thm.~2.11]{GiraRavi86} the variational initial value problem
\begin{equation}
\arraycolsep=2pt
\begin{array}{rcl}
\displaystyle{\tfrac{{\rm d}}{{\rm d}t}\!\int_\varOmega \!\sigma \bm{A}(t)\cdot \varphi\, {\rm d}\xi+\!
	\int_\varOmega \!\nu(\cdot,\|\nabla\!\times\!\bm{A}(t)\|)(\nabla\!\times\!\bm{A}(t))\cdot(\nabla\!\times\! \varphi)\, {\rm d}\xi} &\! =\! &\! 
\displaystyle{\int_\varOmega \chi\, \bm{i}(t)\cdot \varphi\, {\rm d}\xi,} \! \! \!\!\! \!   \\
\displaystyle{\tfrac{{\rm d}}{{\rm d}t}\int_\varOmega \chi^T\bm{A}(t)\, {\rm d}\xi+R\,\bm{i}(t) }& = & \bm{v}(t), \\
\displaystyle{\int_\varOmega \sigma\bm{A}(0)\cdot\varphi\, {\rm d}\xi }& = &
\displaystyle{\int_\varOmega \sigma\bm{A}_0\cdot\varphi\, {\rm d}\xi},\! \! \!\!\! \!\\[3mm]
\displaystyle{\int_\varOmega \chi^T\bm{A}(0)\, {\rm d}\xi }& = &
\displaystyle{\int_\varOmega \chi^T\bm{A}_0\, {\rm d}\xi},
	\end{array}
\label{eq:weak}
\end{equation}
which holds almost everywhere on $[0,T]$ with $T>0$. It follows from  \cite[Thm.~9]{ChiRS23} that under Assumption~\ref{ass:assumptions},
for all $\bm{v}\in L^2(\varOmega;\mathbb{R}^m)$ and 
\mbox{$\bm{A}_0\in X_0(\mbox{curl},\varOmega,\varOmega_C)$}, the coupled MQS system \eqref{eq:MQS} admits a~unique weak solution $(\bm{A},\bm{i})$ on $[0, T]$ in the sense that
\begin{enumerate}[a)]
\item $\sigma\bm{A}\in C([0,T];X(\varOmega,\varOmega_C)) \cap H^1([0,T]; X(\varOmega,\varOmega_C))$,
\item $\int_\varOmega \chi^T\bm{A}(t)\, {\rm d}\xi \in C([0,T];\mathbb{R}^m)\cap H^1([0,T];\mathbb{R}^m)$,
\item $\bm{A}\in L^2([0,T];X_0(\curl,\varOmega,\varOmega_C))$ and $\bm{i}\in L^2([0,T];\mathbb{R}^m)$,
\item equations \eqref{eq:weak} are fulfilled for all $\varphi\in X_0(\mbox{curl},\varOmega,\varOmega_C)$ and almost all $t\in[0,T]$.
\end{enumerate}

\subsection{Passivity}
\label{ssec:passivity}

Passivity is an~important systems-theoretic property of dynamical systems which addresses its energetic behavior. Passive systems are particularly useful in interconnected control design and network synthesis \cite{AndeV73, Will72}. Such systems have the property that they do not generate energy on their own. Passivity of the coupled MQS system~\eqref{eq:MQS},~\eqref{eq:output} has been studied extensively in \cite{ReiS23}. Our passivity analysis of model reduction methods relies heavily on the results presented there which will be reviewed below.

Let us consider first a~general (possibly infinite-dimensional) DAE control system
\begin{subequations}\label{eq:inf_sys}
\begin{align}
    \tfrac{\rm d}{{\rm d}t} \cEl\, x &= \cAl(x) + \cBl\, u, \qquad \cEl\, x(0)=\cEl\, x_0, \label{eq:inf_sys1}\\
     y&=\cCl\, x,\label{eq:inf_sys2}
\end{align}
\end{subequations}
where $\cAl:\mathcal{X}_1\to \mathcal{Z}$ is a~nonlinear continuous operator, and 
$\cEl:\mathcal{X}\to \mathcal{X}$,  
$\cBl:\mathcal{U}\to \mathcal{Z}$ and 
$\cCl:\mathcal{X}\to\mathcal{U}'$ are linear bounded operators 
acting on the Hilbert spaces $\mathcal{Z}$, $\mathcal{U}$, $\mathcal{X}$ and $\mathcal{X}_1$ with continuous embedding $\mathcal{X}_1\subset \mathcal{X}$. The input $u\in L^2([0,T];\mathcal{U})$ is called {\em admissible with the initial condition} $\cEl x(0)=\cEl x_0$ if the initial value problem~\eqref{eq:inf_sys} has a~solution $x:[0,T]\to\mathcal{X}$ such that the state equation in \eqref{eq:inf_sys1} holds in the sense of weak derivatives (in particular, $x$ is continuous as a~function to $\mathcal{X}$, and locally integrable as a~function to $\mathcal{X}_1$) and the output satisfies $y\in L^2([0,T];\mathcal{U}')$. 

\begin{definition}[Passivity]
	{\em A~function $\mathcal{S}: \mathcal{X}\to\mathbb{R}_{\ge0}\cup\{\infty\}$ is called a~{\em sto\-rage function for passivity} of the DAE system \eqref{eq:inf_sys}, if
		for all $T>0$, $x_0\in\mathcal{X}$ with \mbox{$\mathcal{S}(x_0)<\infty$} and all inputs $u\in L^2([0,T]; \mathcal{U}')$ admissible with the initial condition
		$\cEl x(0)=\cEl x_0$, the following conditions are fulfilled:
		\begin{enumerate}[\rm a)]
			\item $t\mapsto \mathcal{S}(x(t))$ is continuous as a function from $[0,T]$ to $\mathbb{R}_{\ge0}\cup\{\infty\}$;
			\item for all $0\leq t_0\leq t_1 \leq T$, the output $y \in L^2([0,T];\mathcal{U})$ fulfills the {\em dissipation inequa\-li\-ty}
			\begin{equation}
			\mathcal{S}(x(t_1))-\mathcal{S}(x(t_0))\leq \int_{t_0}^{t_1} \big\langle u(\tau), y(\tau)\big\rangle \,{\rm d}\tau.
			\label{eq:pass}
			\end{equation}
		\end{enumerate}
The DAE system \eqref{eq:inf_sys} is called {\em passive}, if there exists a storage function for passivity.
}
\end{definition}

The dissipation inequality \eqref{eq:pass} means that the stored energy $\mathcal{S}(x(t_1))$ at any time $T\geq t_1>t_0$ does not exceed the sum of the stored energy $\mathcal{S}(x(t_0))$ at time $t_0\geq 0$ and the total energy $\int_{t_0}^{t_1} \big\langle u(\tau), y(\tau)\big\rangle\,{\rm d}\tau$. Note that the introduced notion of passivity involves 
the state-space representation of the control system. It is also possible to define passivity as an~input-output property.

\begin{definition}
{\em A~DAE control system \eqref{eq:inf_sys} with $\cEl x_0=0$ is called {\em input-output-passive} ({\em io-passive}) if for all $0\leq t\leq T$ and 
all inputs \mbox{$u\in L^2([0,T];\mathcal{U})$} admissible with the initial condition $\cEl\,x(0)=0$, the output $y\in L^2([0,T];\mathcal{U}')$ satisfies
\[
\int_{0}^{t} \big\langle u(\tau),y(\tau)\big\rangle \,{\rm d}\tau \geq 0.
\]
}
\end{definition}

\begin{remark}\label{rem:iopass}
Note that if $\mathcal{S}(0)=0$, then passivity of \eqref{eq:inf_sys} immediately implies \mbox{io-pas}sivity \textup{\cite[Rem.~3.2~b)]{ReiS23}}. The reverse statement can be established under an~additional assumption of reachability \textup{\cite{Brue2010,HillM80}}.
\end{remark}

We now show that the coupled MQS system \eqref{eq:MQS} together with the output equation~\eqref{eq:output} fits into the framework \eqref{eq:inf_sys}. 
To this end, we define the operators
\begin{equation*}
\arraycolsep=2pt
\begin{array}{rrcl}
\qquad\qquad \cEl:&X(\varOmega,\varOmega_C)\times \mathbb{R}^m\,&\to& X(\varOmega,\varOmega_C) \times \mathbb{R}^m,\qquad\qquad\qquad\\
&(\bm{A},\bm{i})\,&\mapsto&(\sigma\bm{A}, \int_{\varOmega} \chi^T\bm{A}\,{\rm d}\xi),\\[3mm]
\cAl:& X_0(\mathrm{curl},\varOmega,\varOmega_C)\times \mathbb{R}^m\,&\to& X_0(\curl,\varOmega,\varOmega_C)'\times \mathbb{R}^m,\\
&(\bm{A},\bm{i})\,&\mapsto&(-\cAl_{11}(\bm{A})+ \chi\,\bm{i},-R\,\bm{i}),\\[3mm]
\cBl:& \mathbb{R}^m\,&\to& X_0(\curl,\varOmega,\varOmega_C)'\times \mathbb{R}^m,\\
&\bm{v}\,&\mapsto&(0,\bm{v}), \\[3mm]
\cCl :& X(\varOmega,\varOmega_C)\times \mathbb{R}^m\,&\to&\mathbb{R}^m,\\
& (\bm{A},\bm{i})\,&\mapsto & \bm{i},
\end{array}
\end{equation*}
with a nonlinear operator
\begin{equation}\label{eq:A11}
\arraycolsep=2pt
\begin{array}{rrcl}
\cAl_{11}:& X_0(\mathrm{curl},\varOmega,\varOmega_C)&\to&\,X_0(\mathrm{curl},\varOmega,\varOmega_C)',\\
&\bm{A}&\mapsto&\displaystyle{\,\left(\varphi\mapsto
\int_\varOmega \nu(\cdot,\|\nabla\times \bm{A}\|)(\nabla\times \bm{A})\cdot(\nabla\times \varphi)\, {\rm d}\xi\right)}.
\end{array}
\end{equation}
Then the MQS system \eqref{eq:MQS},~\eqref{eq:output} can be written as the abstract DAE control system \eqref{eq:inf_sys} with the input $u=\bm{v}$, the state $x=(\bm{A},\bm{i})$, 
the output $y=\bm{i}$, and the initial condition $x_0=(\bm{A}_0, 0)$. 
It follows from \cite[Thm.~9]{ChiRS23} that for 
\mbox{$\bm{A}_0\in X_0(\mathrm{curl},\varOmega,\varOmega_C)$}, any input $\bm{v}\in L^2([0,T]; \mathbb{R}^m)$ is admissible with 
the initial condition  
$\cEl\,(\bm{A}(0), \bm{i}(0))=\cEl\,(\bm{A}_0, 0)$. Then following \cite{ReiS23}, we define 
a~storage function $\mathcal{S}: X_0(\mathrm{curl},\varOmega,\varOmega_C)\times \mathbb{R}^m\to\mathbb{R}_{\geq 0}$ for the MQS system~\eqref{eq:MQS},~\eqref{eq:output} as the magnetic energy
\[
\mathcal{S}(\bm{A},\bm{i})=\int_{\varOmega} \vartheta\bigl(\xi,\|(\nabla\times \bm{A})(\xi)\|^2\bigr)\,{\rm d}\xi 
\]
with the magnetic energy density
\begin{equation}
\vartheta(\xi,\|(\nabla\times \bm{A})(\xi)\|^2)=
\int_0^{\|(\nabla\times \bm{A})(\xi)\|} \nu(\xi,\zeta)\zeta\,{\rm d}\zeta.
\label{eq:gamma}
\end{equation}
Due to \cite[Thm.~3.3]{ReiS23}, we obtain that for all $\bm{v}\in L^2([0,T]; \mathbb{R}^m)$ and all \mbox{$0\leq t_0\leq t_1\leq T$},
the solution $(\bm{A},\bm{i})$ of \eqref{eq:MQS},~\eqref{eq:output} satisfies the energy balance equation
\[
 \mathcal{S}(\bm{A}(t_1),\bm{i}(t_1))-\mathcal{S}(\bm{A}(t_0),\bm{i}(t_0)) = -\int_{t_0}^{t_1} d(\bm{A}(\tau),\bm{i}(\tau)) \,{\rm d}\tau+
\int_{t_0}^{t_1} \bm{v}^T\!(\tau)\, \bm{i}(\tau)\,{\rm d}\tau
\]
with the dissipation function 
\[
d(\bm{A}(\tau),\bm{i}(\tau)) = \big\|\tfrac{\rm d}{{\rm d}\tau} \sqrt{\sigma}\bm{A}(\tau))\big\|^2_{L^2(\varOmega;\mathbb{R}^3)}+\bm{i}^T\!(\tau)\,R\,\bm{i}(\tau)
\] 
which determines the power dissipation. Assumption~\ref{ass:assumptions}~\ref{ass:resistance}) implies that \mbox{$d(\bm{A}(\tau),\bm{i}(\tau))\!\geq\! 0$} for all $\tau\in[0,T]$. Then the dissipation inequality \eqref{eq:pass} is fulfilled and thus the MQS system \eqref{eq:MQS},~\eqref{eq:output} is passive. 
Due to $\mathcal{S}(0)=0$, it is also io-passive. 

\section{Properties of the FEM model}
\label{sec:fem}

In this section, we briefly discuss the spatial discretization of the MQS system \eqref{eq:MQS} by using the FEM and present a~new regularization approach for the resulting FEM model. We also study the structural properties and passivity of this model.

\subsection{FEM discretization}

For the FEM discretization on the 3D domain $\varOmega$, we employ the $H(\mbox{curl},\varOmega)$-con\-for\-ming N\'ed\'elec elements of first type \cite{Ned80}, which are also known as edge elements or Whitney elements of first type, see \cite[Sect.~5]{Boss98}. Let~$\mathcal{T}_h(\varOmega)$ 
be a~regular simplicial triangulation of the domain $\varOmega$, and let $n_n$, $n_e$ and~$n_f$ be the number of nodes, edges and faces in $\mathcal{T}_h(\varOmega)$, respectively. Furthermore, let 
\mbox{$\varphi_1^e,\ldots,\varphi_{n_e}^e\in H_0(\curl,\varOmega)\subset X_0(\curl,\varOmega,\varOmega_C)$} be the edge basis functions which are continuous inside the elements and 
tangentially continuous at the element interfaces.
Approximating the magnetic vector potential~$\bm{A}$ and the initial magnetic vector potential $\bm{A}_0$  by the linear combinations
\[
\bm{A}(\xi,t)\approx \sum_{k=1}^{n_e} \alpha_k^e(t)\,\varphi_k^e(\xi), 
\qquad \bm{A}_0(\xi)\approx \sum_{k=1}^{n_e} \alpha_{k,0}^e\,\varphi_k^e(\xi),
\]
respectively, we obtain from the weak formulation \eqref{eq:weak} by Galerkin projection 
a~quasilinear finite-dimensional DAE system 
\begin{equation}\label{eq:nlDAE}
\arraycolsep=2pt
\begin{array}{rcll}
 \tfrac{\rm d}{{\rm d}t} M_\sigma a & = & -K(a)a+X\bm{i}, & \qquad M_\sigma a(0)=M_\sigma a_0, \\[2mm]
 \tfrac{\rm d}{{\rm d}t} X^T a & = & \hspace*{14mm} -R\,\bm{i} + \bm{v}, & \qquad X^T a(0)=X^Ta_0.
 \end{array}
\end{equation}
Here, $a=[\alpha_1^e,\ldots,\alpha_{n_e}^e]^T$ and $a_0=[\alpha_{1,0}^e,\ldots,\alpha_{n_e,0}^e]^T$ are the semidiscretized magnetic vector potential and initial vector, respectively,
$M_\sigma\in\mathbb{R}^{n_e\times n_e}$ is a~conductivity matrix, 
\mbox{$K(a)\in\mathbb{R}^{n_e\times n_e}$} is a~curl-curl matrix, and
$X\in\mathbb{R}^{n_e\times m}$ is a~coupling matrix with entries
\begin{align}
(M_\sigma)_{kl} & = \int_{\varOmega} \sigma(\xi)\, \varphi_l^e(\xi)\cdot\varphi_k^e(\xi)\, d\xi, \nonumber\\[1mm]
(K(a))_{kl} & = \int_{\varOmega} 
\nu\bigl(\xi,\|\nabla\times \sum_{j=1}^{n_e} \alpha_j^e\varphi_j^e(\xi)\|\bigr)
\big(\nabla\times \varphi_l^e(\xi)\big)\cdot\big(\nabla\times \varphi_k^e(\xi)\big)\, d\xi, \label{eq:K}\\[1mm]
(X)_{kj} & = \int_{\varOmega} \chi_{j}(\xi)\cdot\varphi_k^e(\xi)\, d\xi, \nonumber
\end{align}
respectively, for $k,l=1,\ldots, n_e$ and $j=1,\ldots, m$. 
Reordering the basis functions $\varphi_k^e$ according to the conducting and non-conducting subdomains, we obtain the partitions 
\mbox{$a=[a_1^T, \; a_2^T]^T$}, $a_0=[a_{1,0}^T,\; a_{2,0}^T]^T$, and 
\begin{equation}\label{eq:blockmatr}
M_\sigma = \begin{bmatrix} M_{11} & 0\\ 0 & 0 \end{bmatrix}, \qquad
K(a) = \begin{bmatrix} K_{11}(a_1) & K_{12}\\ K_{21}      & K_{22} \end{bmatrix}, \qquad
X=\begin{bmatrix} X_1 \\ X_2 \end{bmatrix},
\end{equation}
where $a_1,a_{1,0}\in\mathbb{R}^{n_1}$ and $a_2, a_{2,0}\in\mathbb{R}^{n_2}$ with $n_e=n_1+n_2$, and $M_{11}\in\mathbb{R}^{n_1\times n_1}$ is symmetric and positive definite. The latter implies that $M_\sigma$ is symmetric and positive semidefinite. Since the magnetic reluctivity is nonlinear only on the subdomain $\varOmega_{C}$, only the block \mbox{$K_{11}(a_1)\in\mathbb{R}^{n_1\times n_1}$} depends nonlinearly on $a_1$, while other blocks $K_{12}^{}$, $K_{21}^{}=K_{12}^T$ and $K_{22}^{}$ are constant. 
Furthermore, due to Assumption~\ref{ass:assumptions}~\ref{ass:winding})(iii) the matrix $X_2\in\mathbb{R}^{n_2\times m}$ has full column rank. Note that $X_1=0$ if 
$\mathrm{supp}(\chi)\subset \varOmega_I$ meaning that the currents are injected through the contacts in the non-conducting subdomain $\varOmega_I$.

Let $\varphi_1^f,\ldots,\varphi_{n_f}^f$ be the face basis functions and let $C_d\in\mathbb{R}^{n_f\times n_e}$
be a~discrete curl matrix with entries
\[
(C_d)_{kj}=\left\{ \begin{array}{rl} 1, & \mbox{if edge $j$ belongs to face $k$ and has the same orientation}, \\ 
-1, & \mbox{if edge $j$ belongs to face $k$ and has the opposite orientation}, \\ 
0, & \mbox{if edge $j$ does not belong to face $k$}.\end{array}\right.
\]
Then similarly to the linear case \cite{KerS22}, the curl-curl matrix $K(a)$ and the coupling matrix $X$ 
can be written in the factored form 
\begin{equation}
K(a)=C_d^TM_\nu(C_d\, a) C_d, \qquad X=C_d^T\Upsilon,
\label{eq:CMnuC}
\end{equation}
where the entries of the reluctivity matrix $M_\nu(C_d\,a)$ and the matrix $\Upsilon$ are given by
\begin{align*}
(M_\nu(C_d\,a))_{kl} & = \int_{\varOmega} \nu\bigl(\xi,\|\sum_{j=1}^{n_f} \alpha_j^f\varphi_j^f(\xi)\|\bigr)
\varphi_l^f(\xi)\cdot \varphi_k^f(\xi)\, {\rm d}\xi, \quad k,l=1,\ldots, n_f, \\
\Upsilon_{kj} & =\int_\varOmega \gamma_j(\xi)\cdot\varphi_k^f(\xi)\,{\rm d}\xi, \quad\qquad\qquad k=1,\ldots,n_f,\; j=1,\ldots,m.
\end{align*}
Here, $C_d\,a=[\alpha_1^f,\ldots,\alpha_{n_f}^f]^T$, and $\gamma=[\gamma_1,\ldots,\gamma_m]$ is defined in Remark~\ref{rem:reluct_wind}~\ref{rem:wind}). 
Note that $M_\nu(C_d\,a)$ is symmetric and positive definite for all $a\in\mathbb{R}^{n_e}$, and, hence, $K(a)$ is symmetric and positive semidefinite.
 
\begin{remark}
 In the 2D case, the MQS model \eqref{eq:MQS} can be discretized 
using Lagrange nodal elements. As a~result, one gets the FEM model of the same structure as in \eqref{eq:nlDAE}, where, additionally, the matrix $K(a)$ is positive definite for all $a$. 
\end{remark}

Another approach to the spatial discretization of the MQS model~\eqref{eq:MQS} is the 
FIT, e.g., \textup{\cite{CleW01}}. Since the resulting FIT model has the same block structure as the FEM model \eqref{eq:nlDAE}, \eqref{eq:blockmatr}, see \textup{\cite{Schoeps11}}, our subsequent results can be applied to the FIT model in a~similar manner.

\subsection{Regularization of the FEM model}

The FEM model \eqref{eq:nlDAE} with the output equation \eqref{eq:output} can shortly be written as a~DAE control system
\begin{equation}
\arraycolsep=2pt
\begin{array}{rcl}
\tfrac{\rm d}{{\rm d}t}\mathcal{E}_{\FEM}\, x_{\FEM}& = & \mathcal{A}_{\FEM}(x_{\FEM})\, x_{\FEM}+\mathcal{B}_{\FEM}\, u, \quad \mathcal{E}_{\FEM}\, x_{\FEM}(0)=\mathcal{E}_{\FEM}\, x_{\FEM,0}, \\
y_{\FEM} & = & \mathcal{B}_{\FEM}^T\, x_{\FEM},
\end{array}
\label{eq:FEM-DAE}
\end{equation}
with the state $x_{\FEM}=[a^T, \bm{i}^T]^T$, the input $u=\bm{v}$, the output $y_{\FEM}=\bm{i}$, the initial vector \mbox{$x_{\FEM,0}=[a_0^T,\,0]^T$}, and the system matrices 
\begin{equation}
\mathcal{E}_{\FEM} = \begin{bmatrix} M_\sigma & 0\; \\ X^T & 0\;\end{bmatrix}, \qquad
\mathcal{A}_{\FEM}(x_{\FEM}) = \begin{bmatrix} -K(a) & X\; \\ 0 & -R\;\end{bmatrix}, \qquad
\mathcal{B}_{\FEM} = \begin{bmatrix} \;0\; \\ \;I\;\end{bmatrix}.
\label{eq:FEM-DAE-matr}
\end{equation}
In the 2D case, the matrix $K(a)$ is positive definite, and, hence, this system is re\-gular and is of tractability index one \cite{KerBS17}. In the 3D case, the DAE system \eqref{eq:FEM-DAE},~\eqref{eq:FEM-DAE-matr} is, in general, singular since the matrices $\mathcal{E}_{\FEM}$ and $\mathcal{A}_{\FEM}(x_{\FEM})$ might
have a~nontri\-vial common kernel. In terms of the original system in weak formulation \eqref{eq:weak}, this kernel corresponds to the space of all divergence-free elements of $H_0(\mathrm{curl},\varOmega)$ that vanish on~$\varOmega_C$. Note that the latter is a~subspace of $X(\varOmega,\varOmega_C)^\bot$, which is, by definition, the space of all gradient fields of functions from $H^1_0(\varOmega)$ being constant on each interface component $\mathit{\Gamma}_1,\ldots,\mathit{\Gamma}_q$ and $\mathit{\Gamma}_{\rm ext}$.
To overcome the difficulty caused by singularity, system~\eqref{eq:FEM-DAE},~\eqref{eq:FEM-DAE-matr} can be regularized similarly to the infinite-dimensional case using gauging \cite{Hip00,ManC95} or grad-div regularization \cite{Boss01,CleSDGB11,CleW02}. An~alternative approach is to eliminate the over- and underdetermined part using the special structure of the system matrices $\mathcal{E}_{\FEM}$ and $\mathcal{A}_{\FEM}(x_{\FEM})$. Such a~regularization approach has been already applied in \cite{KerS22} to the linear MQS system with constant $\nu_C$. Here, we extend it to the quasilinear system~\eqref{eq:FEM-DAE},~\eqref{eq:FEM-DAE-matr}. 

First, we resolve the second equation in \eqref{eq:nlDAE} for $\bm{i}$ and insert it into the first one. This leads to the DAE system 
\begin{equation}\label{eq:MQSDAE3Dnocoupling}
\tfrac{\rm d}{{\rm d}t}\mathcal{E}a=-K(a)\,a + \mathcal{B}\,u
\end{equation}
with the matrices
\begin{align*}
\mathcal{E}&=\begin{bmatrix}
M_{11}+X_1^{} R^{-1} X_1^T & X_1^{} R^{-1} X_2^T\\[1mm]
X_2^{} R^{-1} X_1^T & X_2^{} R^{-1} X_2^T  
\end{bmatrix}
=\begin{bmatrix}
I &\; C_1^T \Upsilon\\[1mm]
0 &\; C_2^T \Upsilon
\end{bmatrix}
\begin{bmatrix}
M_{11} &0\\[1mm]0& R^{-1}
\end{bmatrix}
\begin{bmatrix}
I &\; 0\\[1mm] \Upsilon^T C_1&\; \Upsilon^T C_2
\end{bmatrix}, \nonumber\\
K(a)&=\begin{bmatrix}
K_{11}(a_1) & K_{12}\\[1mm]
K_{21}      & K_{22}
\end{bmatrix} =\begin{bmatrix} C_1^T\\[1mm] C_2^T \end{bmatrix} M_\nu(C_d\,a)
\begin{bmatrix} C_1 &\; C_2 \end{bmatrix}, \\
\mathcal{B}&=\begin{bmatrix} X_1\\[1mm] X_2  \end{bmatrix} R^{-1} =
\begin{bmatrix} C_1^T \Upsilon \\[1mm] C_2^T \Upsilon \end{bmatrix} R^{-1}, \nonumber
\end{align*}
where $C_d=\begin{bmatrix} C_1 &\; C_2\end{bmatrix}$ is partitioned according to $a=\begin{bmatrix} a_1^T &\; a_2^T\end{bmatrix} ^T$. The output $y_{\FEM}$ takes then the form 
\[
y_{\FEM} = \bm{i} = -R^{-1}\tfrac{\rm d}{{\rm d}t}\begin{bmatrix} X_1^T &\; X _2^T \end{bmatrix} a+ R^{-1} \bm{v} =-\tfrac{\rm d}{{\rm d}t}\mathcal{B}^T a+ R^{-1} u.
\]
It should be noted that if $\mbox{ker}(C_2)$ is nontrivial, then the matrices $\mathcal{E}$ and $K(a)$ have 
a~nontrivial common kernel. It can be determined analogously to the linear case \cite[Thm.~1]{KerS22}
as 
\[
\ker(\mathcal{E})\cap\ker(K(a)) = \mbox{im} \begin{bmatrix} 0 \\ Y_{C_2}\end{bmatrix},
\]
where the columns of $Y_{C_2}\in\mathbb{R}^{n_2\times k_2}$ form a~basis of $\ker(C_2)$. 

In order to find out the over- and underdetermined components of \eqref{eq:MQSDAE3Dnocoupling}, 
we consider a nonsingular matrix 
\begin{equation*}
V = \begin{bmatrix}
I_{n_1} &\enskip 0             &\enskip 0\\
0       &\enskip \hat{Y}_{C_2} &\enskip Y_{C_2}
\end{bmatrix},
\end{equation*}
where the columns of $\hat{Y}_{C_2}\in \mathbb{R}^{n_2\times (n_2-k_2)}$ form a basis of $\mbox{im}(C_2^T)$. Multiplying system~\eqref{eq:MQSDAE3Dnocoupling} from the left with $V^T$ and introducing a~new state vector
\begin{equation}\label{eq:Tinva}
\begin{bmatrix} a_1\\ a_{21}\\ a_{22} \end{bmatrix}= V^{-1} a,
\end{equation}
we obtain an~equivalent DAE system with the transformed system matrices 
\begin{align*}
V^T \mathcal{E} V &=\begin{bmatrix}
M_{11}+ X_1^{} R^{-1}X_1^T &\; X_1^{} R^{-1} X_2^T\hat{Y}_{C_2}^{} &\; 0\\[1mm]
\hat{Y}_{C_2}^T X_2^{} R^{-1} X_1^T &\;  \hat{Y}_{C_2}^T X_2^{} R^{-1} X_2^T \hat{Y}_{C_2}^{} &\; 0\\[1mm]
0&\; 0&\; 0
\end{bmatrix},\\
V^T K(a) V &=\begin{bmatrix}
K_{11}(a_1) &\; K_{12}\hat{Y}_{C_2}^{} &\; 0\\[1mm]
\hat{Y}_{C_2}^T K_{21} &\; \hat{Y}_{C_2}^T K_{22}\hat{Y}_{C_2}^{} &\; 0\\[1mm]
0&\; 0&\; 0
\end{bmatrix},\qquad\quad
V^T \mathcal{B} = \begin{bmatrix}
X_1 \\[1mm]
\hat{Y}_{C_2}^T X_2 \\[1mm] 0
\end{bmatrix}R^{-1}.  
\end{align*}
Further, by introducing the vector $\begin{bmatrix} a_{1,0}^T &\; a_{21,0}^T &\; a_{22,0}^T\end{bmatrix} ^T= V^{-1} a_0$,
the initial condition $\mathcal{E}_{\FEM}\,x_{\FEM}(0)=\mathcal{E}_{\FEM}\,x_0$ can equivalently be written as
\begin{equation}
\begin{bmatrix}
M_{11} & 0 & 0 \\[1mm]
X_1^T    & X_2^T\hat{Y}_{C_2} & 0 
\end{bmatrix} \begin{bmatrix} a_1(0)\\ a_{21}(0)\\ a_{22}(0) \end{bmatrix} = 
\begin{bmatrix}
M_{11} & 0 & 0 \\[1mm]
X_1^T    & X_2^T\hat{Y}_{C_2} & 0 
\end{bmatrix} \begin{bmatrix} a_{1,0}\\ a_{21,0}\\ a_{22,0} \end{bmatrix}.
\label{eq:initFEM}
\end{equation}
One can see that the components of $a_{22}$ are actually not involved in the transformed system and the initial condition. As a~consequence, they can be chosen freely. Removing the trivial equation $0=0$, we obtain a~regular DAE control system 
	\begin{subequations}\label{eq:MQS3Dreg}
	\begin{align}
	\tfrac{\rm d}{{\rm d}t}\mathcal{E}_r\,  x_r&= \mathcal{A}_r(x_r)\,x_r+\mathcal{B}_r\, u,\label{eq:MQS3Dregevul}\\
	y_r&=-\tfrac{\rm d}{{\rm d}t}\mathcal{B}_r^T x_r+R^{-1}u,\label{eq:MQS3Dregout}
\end{align}
	\end{subequations}
where $x_r=\begin{bmatrix} a_1^T,\, a_{21}^T\end{bmatrix}^T$, and 
\begin{equation}\label{eq:MQS3Dregmat}
\begin{aligned}
\mathcal{E}_r&=
\begin{bmatrix}
M_{11}+ X_1^{} R^{-1}X_1^T &\; X_1^{} R^{-1} X_2^T\hat{Y}_{C_2}^{}\\[1mm]
\hat{Y}_{C_2}^T X_2^{} R^{-1} X_1^T &\;  \hat{Y}_{C_2}^T X_2^{} R^{-1} X_2^T \hat{Y}_{C_2}^{} 
\end{bmatrix}\in \mathbb{R}^{n_r\times n_r},\\
\mathcal{A}_r(x_r) &=-
\begin{bmatrix}
K_{11}(a_1) &\; K_{12}\hat{Y}_{C_2}^{} \\[1mm]
\hat{Y}_{C_2}^T K_{21} &\; \hat{Y}_{C_2}^T K_{22}\hat{Y}_{C_2}^{} 
\end{bmatrix}\in \mathbb{R}^{n_r\times n_r},\\
\mathcal{B}_r&= 
\begin{bmatrix}
X_1 \\[1mm]
\hat{Y}_{C_2}^T X_2 
\end{bmatrix}R^{-1}\in \mathbb{R}^{n_r\times m} 
\end{aligned}
\end{equation}
with $n_r=n_1+n_2-k_2$. Note that this system has the same output as \eqref{eq:FEM-DAE}, i.e., \mbox{$y_r=y_{\FEM}$}. The regularity of the matrix pencil $\lambda\mathcal{E}_r-\mathcal{A}_r(x_r)$ (and also of the DAE system~\eqref{eq:MQS3Dregevul})
follows from the symmetry of $\mathcal{E}_r$ and $\mathcal{A}_r(x_r)$  and the fact that 
\mbox{$\ker(\mathcal{E}_r)\cap \ker(\mathcal{A}_r(x_r))=\{0\}$}. Using \eqref{eq:CMnuC}, the system matrices in \eqref{eq:MQS3Dregmat} can also be written in the short form
\begin{align}
\mathcal{E}_r &=\mathcal{F}_\sigma^{} M_{\sigma,R}^{} \mathcal{F}_\sigma^T,\qquad
\mathcal{A}_r(x_r)=-\mathcal{F}_\nu^{} M_\nu^{}(\mathcal{F}_\nu^T x_r)\mathcal{F}_\nu^T,\label{eq:EregAregfac}\\
\mathcal{B}_r &=\mathcal{F}_\nu \Upsilon R^{-1} = \mathcal{F}_\sigma M_{\sigma,R} \begin{bmatrix}0\\I_m\end{bmatrix},\label{eq:Bregfac}
\end{align}
where
\begin{align*}
\mathcal{F}_\sigma\!=\!\begin{bmatrix}
I_{n_1} &\; X_1\\0&\; \hat{Y}_{C_2}^T X_2
\end{bmatrix}
\!=\!\begin{bmatrix}
I_{n_1} & \; C_1^T \Upsilon\\
0 &\; \hat{Y}_{C_2}^T C_2^T \Upsilon
\end{bmatrix},
\quad
M_{\sigma,R}=\begin{bmatrix}
M_{11}&0\\0& R^{-1}
\end{bmatrix},\quad
\mathcal{F}_\nu=\!\begin{bmatrix} C_1^T \\ \hat{Y}_{C_2}^T C_2^T \end{bmatrix}. 
\end{align*}
This shows that $\mathcal{E}_r$ is positive semidefinite and $\mathcal{A}_r(x_r)$ is negative semi\-de\-fi\-ni\-te, since $M_{11}$, $R$ and $M_\nu(\mathcal{F}_\nu^T x_r)$ are positive definite. Furthermore, the initial condition \eqref{eq:initFEM} takes the form 
\begin{equation}
\mathcal{F}_\sigma^T x_r(0)=\mathcal{F}_\sigma^Tx_{r,0}
\label{eq:initreg}
\end{equation}
with $x_{r,0}=\begin{bmatrix} a_{1,0}^T &\; a_{21,0}^T\end{bmatrix}^T$.

In order to investigate the tractability index of the regularized DAE system \eqref{eq:MQS3Dreg}, we employ an~admissible matrix function sequences approach from \cite[Chapt.~3]{LamMT13}. 
Let $h(x_r)=\mathcal{A}_r(x_r)x_r$, $\mathcal{J}_h(x_r)$ denote the Jacobian matrix of $h$ at $x_r$, and let $\mathcal{Q}$ be a~projector onto $\ker(\mathcal{E}_r)$. By definition \cite[Def.~3.28]{LamMT13}, the DAE system \eqref{eq:MQS3Dreg} has tractability index one if the matrix $\mathcal{G}_1(x_r)=\mathcal{E}_r-\mathcal{J}_{h}(x_r)\mathcal{Q}$ is nonsingular for all $x_r$. The following theorem shows that the regularized DAE system \eqref{eq:MQS3Dreg}, \eqref{eq:MQS3Dregmat} indeed has this property.

\begin{theorem}\label{Thm:MQS3Dindex1}
	Consider the DAE \eqref{eq:MQS3Dreg}, \eqref{eq:MQS3Dregmat}, where $M_{11}$ and $M_\nu(C_d\,a)$ are symmetric and positive definite, $X_2$ has full column rank, and the columns of $\hat{Y}_{C_2}$ form a~basis of $\mbox{\rm im}(C_2^T)$. 
	This system has tractability index one. 
\end{theorem}

\begin{proof}
	Let the columns of $Y$ form an~orthonormal basis of $\ker(X_2^T\hat{Y}_{C_2})$. Then a~projector $\mathcal{Q}$ onto $\ker(\mathcal{E}_r)$ can be chosen as
	\begin{equation*}
	\mathcal{Q}=\begin{bmatrix} 0&\; 0\\ 0&\; YY^T\\ \end{bmatrix}.
	\end{equation*}
	In this case, the matrix
	\[
	\mathcal{G}_1(x_r) = \mathcal{E}_r -\mathcal{J}_h(x_r)\mathcal{Q}
	=\begin{bmatrix}
	M_{11}+X_1^{}R^{-1} X_1^T & X_1^{} R^{-1} X_2^T \hat{Y}_{C_2}^{}+K_{12}\hat{Y}_{C_2}^{}Y Y^T\\[1mm]
	\hat{Y}_{C_2}^T X_2^{} R^{-1} X_1^T &  \hat{Y}_{C_2}^T X_2^{} R^{-1} X_2^T\hat{Y}_{C_2}^{} + 
	\hat{Y}_{C_2}^T K_{22} \hat{Y}_{C_2}^{} Y Y^T
	\end{bmatrix}
	\]
	is independent of $x_r$. We now show that this matrix is nonsingular. Assume that there exists a~vector 
	$w=\begin{bmatrix} w_1^T &\; w_2^T\end{bmatrix}^T$ such that $\mathcal{G}_1w=0$. This equation implies  that
	\begin{align}
	(M_{11}+X_1^{}R^{-1} X_1^T)w_1+(X_1^{} R^{-1} X_2^T \hat{Y}_{C_2}^{}+K_{12}\hat{Y}_{C_2}^{}Y Y^T)w_2&=0,\label{eq:3DMtrac1}\\
	\hat{Y}_{C_2}^T X_2^{} R^{-1} X_1^T w_1+(\hat{Y}_{C_2}^T X_2^{} R^{-1} X_2^T\hat{Y}_{C_2}^{} + 
	\hat{Y}_{C_2}^T K_{22} \hat{Y}_{C_2}^{} Y Y^T)w_2&=0.\label{eq:3DMtrac2}
	\end{align}
	Multiplying equation \eqref{eq:3DMtrac2} from the left with $Y^T$ and using $Y^T\hat{Y}_{C_2}^T X_2=0$, we obtain that
	$Y^T \hat{Y}_{C_2}^T K_{22} \hat{Y}_{C_2}^{} Y Y^Tw_2=0$. 
	Since $M_\nu(C_d\,a)$ is symmetric, positive definite and~$\hat{Y}_{C_2}$ is a~basis of $\mbox{im}(C_2^T)$,  the matrix $\hat{Y}_{C_2}^T K_{22} \hat{Y}_{C_2}^{}=\hat{Y}_{C_2}^TC_2^T M_\nu(C_d\,a) C_2^{} \hat{Y}_{C_2}^{}$ is symmetric, positive definite,
	and, hence, $YY^Tw_2=0$. Using the fact that $Y$ has full column rank, we obtain that $Y^Tw_2=0$.
	
	Next, we show that $\hat{Y}_{C_2}^T X_2$ has full column rank, and, hence $X_2^T \hat{Y}_{C_2}^{}\hat{Y}_{C_2}^T X_2$ is nonsingular. Indeed, let $\hat{Y}_{C_2}^T X_2v=0$ for a~vector $v\in\mathbb{R}^m$. 
	Then $X_2v\in \ker(\hat{Y}_{C_2}^T)$. On the other hand, $X_2v=C_2^T\Upsilon v\in \mbox{im}(C_2^T)=\mbox{im}(\hat{Y}_{C_2})$.
	Therefore, $X_2v=0$. Since $X_2$ has full column rank, we get $v=0$.
	
	Multiplying equation \eqref{eq:3DMtrac2} from the left with $X_1(X_2^T\hat{Y}_{C_2}^{}\hat{Y}_{C_2}^T X_2)^{-1} X_2^T\hat{Y}_{C_2}^{}$ and using \mbox{$Y^Tw_2=0$}, we have
	\begin{equation}\label{eq:3DMtrac2'}
	X_1^{} R^{-1} X_1^T w_1+ X_1^{} R^{-1} X_2^T\hat{Y}_{C_2}^{}w_2=0.
	\end{equation}
	Subtracting equation \eqref{eq:3DMtrac2'} from \eqref{eq:3DMtrac1} and using again $Y^Tw_2=0$, we obtain that
	$M_{11}w_1=0$ or, equivalently, $w_1=0$. Furthermore,
	multiplying \eqref{eq:3DMtrac2} from left with~$w_2^T$ and using $Y^Tw_2=0$ and $w_1=0$, we have $w_2^T \hat{Y}_{C_2}^T X_2^{} R^{-1} X_2^T\hat{Y}_{C_2}^{}w_2=0$. Since $R$ is symmetric, positive definite,
	$w_2\in\ker(X_2^T\hat{Y}_{C_2}^{})=\mbox{im}(Y)$.  This means that~$w_2$ belongs to the image of $Y$ and also to the kernel of $Y^T$. Therefore, $w_2=0$. 
	Thus, $\mathcal{G}_1$ is nonsingular, and, hence, the DAE system \eqref{eq:MQS3Dreg},~\eqref{eq:MQS3Dregmat} is of tractability index~one. 
	\end{proof}

Our goal is now to transform the output equation \eqref{eq:MQS3Dregout} into the standard form 
$y_r=\mathcal{C}_r x_r$ with an~output matrix $\mathcal{C}_r\in \mathbb{R}^{m\times n_r}$. 
For this purpose, we transform the pencil $\lambda \mathcal{E}_r -\mathcal{A}_r(x_r) $ into a~condensed form which allows us to extract the algebraic constraints in \eqref{eq:MQS3Dregevul} and derive the output matrix~$\mathcal{C}_r$.

\begin{theorem} \label{thm:existTWEi}
Let the matrices $\mathcal{E}_r$ and $\mathcal{A}_r(x_r)$ be as in \eqref{eq:EregAregfac}. Then there exists a~nonsingular constant matrix $W\in\mathbb{R}^{n_r\times n_r}$ such that
\begin{equation}
W^T\mathcal{E}_r W=\begin{bmatrix}
E_{11} &\; 0 &\; 0\\ 0&\; I_{n_0} &\; 0\\0 &\; 0 &\; 0
\end{bmatrix},\qquad                                    
W^T\mathcal{A}_r(x_r)W=\begin{bmatrix}
A_{11}(x_r) &\; 0 &\; 0\\0 &\; 0 &\; 0\\ 0 &\; 0 &\; I_{n_\infty}
\end{bmatrix},
\label{eq:MQSWeier}
\end{equation}
where $E_{11}\in \mathbb{R}^{n_s\times n_s}$ and $-A_{11}(x_r)\in\mathbb{R}^{n_s\times n_s}$ are both symmetric and positive definite, and
$n_s+n_0+n_\infty=n_r$.
\end{theorem}

\begin{proof}
	Let the columns of $Y_{\sigma}\in\mathbb{R}^{n_r\times n_\infty}$ and $Y_{\nu}\in\mathbb{R}^{n_r\times n_0}$ form the bases of $\ker(\mathcal{F}_{\sigma}^T)$ and $\ker(\mathcal{F}_{\nu}^T)$, respectively, i.e.,
	$\mbox{im}(Y_{\sigma})=\ker(\mathcal{F}_{\sigma}^T)$ and $\mbox{im}(Y_{\nu})=\ker(\mathcal{F}_{\nu}^T)$.
	First, note that $Y_{\sigma}^T\mathcal{A}_r(x_r)$ is independent of $x_r$.
	This follows from the fact that the basis matrix $Y_{\sigma}$ has the form 
	$\begin{bmatrix} 0 &\; Y^T\end{bmatrix}^T$, where
	the columns of $Y$ form a~basis of $\ker(X_2^T\hat{Y}_{C_2})$. In this case, we have
	\begin{equation*}
	Y_{\sigma}^T\mathcal{A}_r(x_r)=-\begin{bmatrix} 0 &\; Y^T\end{bmatrix}\begin{bmatrix}
	K_{11}(a_1) & K_{12}\hat{Y}_{C_2} \\
    \hat{Y}_{C_2}^T K_{21}^{} & \hat{Y}_{C_2}^T K_{22}^{}\hat{Y}_{C_2}^{} 
	\end{bmatrix}=-Y^T\begin{bmatrix} \hat{Y}_{C_2}^T K_{21}^{} &\enskip \hat{Y}_{C_2}^T K_{22}^{}\hat{Y}_{C_2}^{} \end{bmatrix},
	\end{equation*}
	which is independent of $x_r$. Therefore, in the following, we just write $Y_{\sigma}^T\mathcal{A}_r$ or $\mathcal{A}_rY_{\sigma}$.
	
	The latter observation allows us to construct a constant transformation matrix~$W$ analogously to the linear case \cite[Thm.~2]{KerS22} as
	\begin{equation*}
	W=\begin{bmatrix}
	W_1  &\enskip Y_{\nu}(Y_{\nu}^T\mathcal{E}_r Y_{\nu})^{-1/2} &\enskip Y_{\sigma}(Y_{\sigma}^T\mathcal{A}_r Y_{\sigma})^{-1/2}
	\end{bmatrix},
	\end{equation*}
	where the columns of $W_1$ form a basis of $\mbox{ker}\left([\mathcal{E}_rY_\nu \enskip\; \mathcal{A}_rY_\sigma]^T\right)$. Using $\mathcal{F}_{\sigma}^TY_{\sigma}^{}=0$, $\mathcal{F}_{\nu}^TY_{\nu}^{}=0$, $W_1^T \mathcal{E}_rY_\nu=0$, and $W_1^T \mathcal{A}_rY_\sigma=0$, we obtain \eqref{eq:MQSWeier}, where
	\begin{equation*}
	E_{11}=W_1^T\mathcal{E}_r W_1,\qquad A_{11}(x_r)=W_1^T\mathcal{A}_r(x_r) W_1
	\end{equation*}
	are symmetric and $E_{11}$ and $-A_{11}(x_r)$ are positive definite. These properties immediately follow from the symmetry and positive semidefiniteness of $\mathcal{E}_r$ and $-\mathcal{A}_r(x_r)$ and the relations
	\begin{align*}
	\rank(W^T\mathcal{E}_r W)     & =\rank(\mathcal{E}_r )= n_r-n_\infty,\\
	\rank(W^T\mathcal{A}_r(x_r)W) & =\rank(\mathcal{A}_r(x_r))=n_r-n_0.
	\end{align*}
	This completes the proof. 
\end{proof}

Note that the input matrix $\mathcal{B}_r$ in \eqref{eq:Bregfac} can also be presented as 
\begin{align}
\mathcal{B}_r&=\mathcal{F}_{\sigma}M_{\sigma,R}
\begin{bmatrix}
0 \\ I_m
\end{bmatrix}=\mathcal{F}_{\sigma}M_{\sigma,R}
\begin{bmatrix}
I_{n_1} & 0\\ X_1^T& X_2^T\hat{Y}_{C_2}
\end{bmatrix}
\begin{bmatrix}
0\\
\hat{Y}_{C_2}^T X_2^{} (X_2^T\hat{Y}_{C_2}^{} \hat{Y}_{C_2}^T X_2^{})^{-1}
\end{bmatrix}\notag\\
&=\mathcal{F}_{\sigma}M_{\sigma,R}\mathcal{F}_{\sigma}^T\begin{bmatrix} 0 \\ \hat{Z}\end{bmatrix} 
= \mathcal{E}_r\begin{bmatrix} 0 \\ \hat{Z}\end{bmatrix}\label{eq:MQS3DBr2}
\end{align}
with $\hat{Z}=\hat{Y}_{C_2}^T X_2^{}(X_2^T\hat{Y}_{C_2}^{}\hat{Y}_{C_2}^T X_2^{})^{-1}$. We consider now a~pseudoinverse of $\mathcal{E}_r$ given by
\[
\mathcal{E}_r^-=W\begin{bmatrix}
E_{11}^{-1} & &\\&I_{n_0} &\\&&0
\end{bmatrix} W^T.
\]
Simple calculations show that this matrix satisfies
\begin{align}
\mathcal{E}_r^{}\mathcal{E}_r^-\mathcal{E}_r^{}& = \mathcal{E}_r^{},\qquad\quad
\mathcal{E}_r^-\mathcal{E}_r^{}\mathcal{E}_r^-=\mathcal{E}_r^-, \qquad\quad
(\mathcal{E}_r^-)^T=\mathcal{E}_r^-,\label{eq:3DEinv}\\
\mathcal{E}_r^-\mathcal{E}_r^{} & = I-\Pi_\infty, \qquad 
\mathcal{E}_r^{}\mathcal{E}_r^-=I-\Pi_\infty^T, \label{eq:3DEinvE}\\
\mathcal{E}_r^{}\mathcal{E}_r^-\mathcal{A}_r(x_r) & = \mathcal{A}_r(x_r)\mathcal{E}_r^-\mathcal{E}_r^{}=\mathcal{E}_r^{}\mathcal{E}_r^-\mathcal{A}_r(x_r)\mathcal{E}_r^-\mathcal{E}_r^{},
\label{eq:3DEEinvA}
\end{align}
where $\Pi_\infty=Y_{\sigma}(Y_{\sigma}^T\mathcal{A}_rY_{\sigma})^{-1}Y_{\sigma}^T\mathcal{A}_r$ is the projector 
onto the right deflating subspace of $\lambda \mathcal{E}_r-\mathcal{A}_r(x_r)$ corresponding to the eigenvalue at infinity.
Equations \eqref{eq:3DEinv} imply that~$\mathcal{E}_r^-$ is the symmetric reflexive inverse of $\mathcal{E}_r$.
Using \eqref{eq:MQS3Dregevul}, \eqref{eq:MQS3DBr2} and the first relation in~\eqref{eq:3DEinv}, the output in \eqref{eq:MQS3Dregout} can be written as 
\begin{align*}
y_r&=-\tfrac{\rm d}{{\rm d}t}\mathcal{B}_r^Tx_r+R^{-1}u
=-\begin{bmatrix} 0 &\; \hat{Z}^T\end{bmatrix}\tfrac{\rm d}{{\rm d}t}\mathcal{E}_r x_r+R^{-1}u\\
&=-\begin{bmatrix} 0 &\; \hat{Z}^T\end{bmatrix}\mathcal{E}_r^{}\mathcal{E}_r^-\tfrac{\rm d}{{\rm d}t}
\mathcal{E}_r^{} x_r+R^{-1}u
=-\mathcal{B}_r^T\mathcal{E}_r^-(\mathcal{A}_r(x_r)x_r+\mathcal{B}_r u)+R^{-1}u\\
&=-\mathcal{B}_r^T\mathcal{E}_r^-\mathcal{A}_r(x_r)x_r+(R^{-1}-\mathcal{B}_r^T\mathcal{E}_r^-\mathcal{B}_r^{})u.
\end{align*}
Taking into account the special block structure of the matrix $\mathcal{A}_r(x_r)$ in \eqref{eq:MQS3Dregmat} and using equations~\eqref{eq:3DEinvE} and~\eqref{eq:3DEEinvA}, we obtain that the matrix
\begin{align}
\mathcal{C}_r & := -\mathcal{B}_r^T\mathcal{E}_r^{-}\mathcal{A}_r(x_r)
=-\begin{bmatrix} 0 &\; \hat{Z}^T\end{bmatrix}\mathcal{A}_r(x_r)(I-\Pi_\infty)\notag\\
& = (X_2^T\hat{Y}_{C_2}\hat{Y}_{C_2}^T X_2)^{-1}X_2^T\hat{Y}_{C_2}\hat{Y}_{C_2}^T\begin{bmatrix}
K_{21} &\enskip K_{22}\hat{Y}_{C_2}
\end{bmatrix}(I-\Pi_{\infty})\label{def:Creg}
\end{align}
is independent of $x_r$. Moreover, it follows from the first equation in \eqref{eq:3DEinv} and \eqref{eq:MQS3DBr2} that 
\begin{align*}
\mathcal{B}_r^T\mathcal{E}_r^-\mathcal{B}_r&=\begin{bmatrix} 0 &\; \hat{Z}^T\end{bmatrix}\mathcal{E}_r\mathcal{E}_r^-\mathcal{E}_r\begin{bmatrix} 0 \\ \hat{Z}\end{bmatrix}
=\begin{bmatrix} 0 &\; \hat{Z}^T\end{bmatrix} \mathcal{F}_{\sigma}M_{\sigma,R}\mathcal{F}_{\sigma}^T\begin{bmatrix} 0 \\ \hat{Z}\end{bmatrix}\\
&=\begin{bmatrix} 0 &\; I_m\end{bmatrix}M_{\sigma,R}\begin{bmatrix} 0 \\ I_m\end{bmatrix}
=R^{-1},
\end{align*}
and, hence, the output takes the form $y_r=\mathcal{C}_rx_r$ with the output matrix $\mathcal{C}_r$ as in~\eqref{def:Creg}. 

Note that the regularized DAE system \eqref{eq:MQS3Dreg},~\eqref{eq:MQS3Dregmat} has the same block structure as that studied in \cite{KerBS17}, which was regularized using gauging. Therefore, we can utilize the transformation method presented there to transform \eqref{eq:MQS3Dreg}, \eqref{eq:MQS3Dregmat}, \eqref{eq:initreg}
into a~control system of ordinary differential equations (ODEs)
\begin{equation}
\arraycolsep=2pt
\begin{array}{rcl}
\tfrac{\rm d}{{\rm d}t} \mathcal{E}_{\ODE}\, x_{\ODE} & = & \mathcal{A}_{\ODE}(x_{\ODE})\, x_{\ODE} + \mathcal{B}_{\ODE}\, u,\quad
x_{\ODE}(0)=x_{\ODE,0},  \\
y_{\ODE} & = & \mathcal{C}_{\ODE}\, x_{\ODE},
\end{array}
\label{eq:ODE}
\end{equation}
with 
$x_{\ODE} = \begin{bmatrix} a_1^T &\enskip (Z^Ta_{21})^T \end{bmatrix}^T$, $x_{\ODE,0}=\begin{bmatrix} a_{1,0}^T &\enskip (Z^Ta_{21,0})^T\end{bmatrix}^T$, $y_{\ODE}=y_r$, and
\begin{equation*}
\arraycolsep=1.4pt
\begin{array}{rcl}
 \mathcal{E}_{\ODE}&=&\enskip\;\begin{bmatrix}   M_{11} +X_1R^{-1}X_1^T & \quad\quad X_1R^{-1}X_2^TZ_2 \\[1mm]
  Z_2^T X_2^{}R^{-1}X_1^T & \quad Z_2^T X_2^{}R^{-1}X_2^TZ_2^{} \end{bmatrix},\\[5mm]
 \mathcal{A}_{\ODE}(x_{\ODE})&\!\!=\!\!&-\! \begin{bmatrix} K_{11}(a_1)& \; K_{12} Z_2\\[1mm]
Z_2^T K_{21}& \; Z_2^T K_{22}Z_2^{} \end{bmatrix}\! +\!
\begin{bmatrix} K_{12}\\[1mm]
Z_2^T K_{22} \end{bmatrix}\!Y_2^{}\!\left(Y_2^T\!K_{22}Y_2^{}\right)^{-1}\!Y_2^T\!
\begin{bmatrix} K_{21} &\; K_{22} Z_2^{} \end{bmatrix}, \\[5mm]
\mathcal{B}_{\ODE}&\!\!=\!\!&\enskip\;\begin{bmatrix} X_1 \\[1mm]
Z_2^T X_2^{}\end{bmatrix}R^{-1},\\[5mm]
\mathcal{C}_{\ODE} & = & \;\hat{Z}^T
\hat{Y}_{C_2}^T \left(I-K_{22}Y_2^{}(Y_2^T K_{22} Y_2^{})^{-1}Y_2^T\right) 
\begin{bmatrix} K_{21} &\; K_{22} Z_2^{}\end{bmatrix}.
\end{array}
\end{equation*}
Here, $Z_2=\hat{Y}_{C_2}Z$ and $Y_2=\hat{Y}_{C_2}Y$, 
where the columns of $Y\in\mathbb{R}^{(n_2-k_2)\times (n_2-k_2-m)}$ form an~orthonormal basis of $\mbox{ker}(X_2^T\hat{Y}_{C_2}^{})$ and
\mbox{$Z\!=\!\hat{Y}_{C_2}^T X_2^{}(X_2^T\hat{Y}_{C_2}^{}\hat{Y}_{C_2}^T X_2^{})^{-1/2}\!\in\!\mathbb{R}^{(n_2-k_2)\times m}$} satisfies $Z^TZ=I$ and $Z^TY=0$. The invertibility 
of $Y_2^TK_{22}Y_2^{}=Y^T\hat{Y}_{C_2}^TK_{22}\hat{Y}_{C_2}^{}Y$ immediately follows from the fact that $\hat{Y}_{C_2}^TK_{22}\hat{Y}_{C_2}^{}$ is symmetric and positive definite which was shown in the proof of Theorem~\ref{Thm:MQS3Dindex1}. Note that $\mathcal{E}_{\ODE}$ is symmetric, positive definite, and $\mathcal{A}_{\ODE}(x_{\ODE})$ is symmetric, negative semidefinite as the Schur complement of the symmetric, negative semidefinite matrix $\mathcal{A}_r(x_r)$.

\subsection{Passivity of the FEM model}

Passivity of the FEM model \eqref{eq:FEM-DAE} and its regularization \eqref{eq:MQS3Dreg} can be established in a~similar way as for the infinite-dimensional system
\eqref{eq:MQS},~\eqref{eq:output} by considering an~appropriate storage function. 
	
\begin{theorem}\label{th:passFEM}
The semidiscretized MQS system \eqref{eq:FEM-DAE}, \eqref{eq:FEM-DAE-matr} is passive.
\end{theorem}

\begin{proof}
Similarly to the infinite-dimensional case in Section~\ref{ssec:passivity}, we define a~storage function 
\[
\mathcal{S}_{\,\FEM}(x_{\FEM}(t))=\int_{\varOmega} \vartheta\bigl(\xi,\|\nabla\times\sum_{j=1}^{n_e}\alpha_j^e(t)\varphi_j^e(\xi)\|^2\bigr)\,{\rm d}\xi,
\]
where $x_{\FEM}=[a^T, \bm{i}^T]^T$ with $a=[\alpha_1^e,\ldots,\alpha_{n_e}^e]^T$ solves \eqref{eq:FEM-DAE} or, equivalently, \eqref{eq:nlDAE}. Then using \eqref{eq:nlDAE} we obtain
\[
\arraycolsep=2pt
\begin{array}{l}
\displaystyle{\textstyle{\frac{\rm d}{{\rm d}t}}\mathcal{S}_{\,\FEM}(x_{\FEM}(t))} =  \displaystyle{
 \int_\varOmega\tfrac{\partial}{\partial\zeta}\vartheta\bigl(\xi,\|\nabla\times\sum_{j=1}^{n_e}\alpha_j^e(t)\varphi_j^e(\xi)\|^2\bigr)
\tfrac{\partial}{\partial t}\|\nabla\times\sum_{j=1}^{n_e}\alpha_j^e(t)\varphi_j^e(\xi)\|^2 \,{\rm d}\xi} \\[2mm]
\quad = \displaystyle{\!\int_\varOmega \nu\bigl(\xi,\|\nabla\!\times\!\sum_{j=1}^{n_e}\alpha_j^e(t)\varphi_j^e(\xi)\|\bigr)
\Bigl(\nabla\times\sum_{k=1}^{n_e}\tfrac{\rm d}{{\rm d}t}{\alpha}_k^e(t)\varphi_k^e(\xi)\Bigr)\cdot
\Bigl(\nabla\times\sum_{l=1}^{n_e}\alpha_l^e(t)\varphi_l^e(\xi)\Bigr){\rm d}\xi} \\[2mm]
\quad = \displaystyle{\!\sum_{k=1}^{n_e}\tfrac{\rm d}{{\rm d}t}{\alpha}_k^e(t)\sum_{l=1}^{n_e}\alpha_l^e(t)
\!\int_\varOmega
\nu\bigl(\xi,\|\nabla\!\times\!\sum_{j=1}^{n_e}\alpha_j^e(t)\varphi_j^e(\xi)\|\bigr)
\Bigl(\nabla\times\varphi_k^e(\xi)\Bigr)\cdot
\Bigl(\nabla\times\varphi_l^e(\xi)\Bigr){\rm d}\xi} \\[6mm]
\quad = \displaystyle{\bigl(\tfrac{\rm d}{{\rm d}t}{a}(t)\bigr)^TK(a(t))a(t)} = 
\displaystyle{-\bigl(\tfrac{\rm d}{{\rm d}t}{a}(t)\bigr)^T\tfrac{\rm d}{{\rm d}t}Ma(t)+\bigl(\tfrac{\rm d}{{\rm d}t}{a}(t)\bigr)^TX\bm{i}(t)}\\[2mm]
\quad = \displaystyle{-\bigl(\tfrac{\rm d}{{\rm d}t}{a}(t)\bigr)^T M\tfrac{\rm d}{{\rm d}t}{a}(t)-\bm{i}^T(t)R\,\bm{i}^{}(t)+\bm{v}^T(t)\bm{i}^{}(t)}\\[2mm]
\quad \leq \bm{v}^T(t)\, \bm{i}(t) = u^T(t)\,y_{\FEM}^{}(t).
\end{array}
\]
The inequality holds since $M$ is positive semidefinite and $R$ is positive definite. Integrating this inequality on $[t_0,t_1]\subseteq [0,T]$, we obtain
the dissipation inequality
\[
\mathcal{S}_{\,\FEM}(x_{\FEM}(t_1))-\mathcal{S}_{\,\FEM}(x_{\FEM}(t_0)) \leq \int_{t_0}^{t_1} u^T(\tau)\, y_{\FEM}^{}(\tau)\, {\rm d}\tau,
\]
which implies the passivity of \eqref{eq:FEM-DAE}, \eqref{eq:FEM-DAE-matr}. 
\end{proof}

\begin{remark}
In the 2D case, the passivity of the semidiscretized model can be proved analogously to Theorem~\textup{\ref{th:passFEM}} by considering a storage function
\[
\mathcal{S}_{\,\FEM}(x_{\FEM}(t))=\int_{\varOmega} \vartheta\bigl(\xi,\|\nabla\sum_{j=1}^{n_a}\alpha_j(t)\varphi_j(\xi)\|^2\bigr)\,{\rm d}\xi,
\]
where $a=[\alpha_1,\ldots,\alpha_{n_a}]^T$ is the discrete magnetic potential, $x_{\FEM}=[a^T,\bm{i}^T]^T$, and $\varphi_j$ are the Lagrange basis functions.
\end{remark}

Since the regularized DAE system \eqref{eq:MQS3Dreg} has the same input-output relation as~\eqref{eq:FEM-DAE}, it is passive. For completeness, we verify this result by constructing an~appropriate storage function.

\begin{theorem}\label{th:passFEMreg}
	The regularized MQS system \eqref{eq:MQS3Dreg} is passive.
\end{theorem}

\begin{proof}
Since $a_{22}$ in \eqref{eq:Tinva} can be chosen arbitrarily and it has no influence on the output $y_r$, we take $a_{22}=0$. Then the discrete magnetic potential~$a$ can be written as $a=\begin{bmatrix}a_1^T&\; a_2^T\end{bmatrix}^T=V_1 \begin{bmatrix}a_1^T &\; a_{21}^T\end{bmatrix}^T$
with
\begin{equation*}
V_1=\begin{bmatrix} I_{n_1} &\; 0 \\ 0&\; \hat{Y}_{C_2}	\end{bmatrix}.
\end{equation*}
Introducing new basis functions
\[
[\psi_1(\xi),\ldots, \psi_{n_r}(\xi)]=[\varphi^e_1(\xi), \ldots,\varphi^e_{n_e}(\xi)]V_1,
\]
we obtain that
\begin{equation}
\sum_{i=1}^{n_{e}}\alpha_i^e(t)\varphi^e_i(\xi)=\sum_{i=1}^{n_{r}}\beta_i(t)\psi_i(\xi),
\label{eq:xr}
\end{equation}
where 
\[
\arraycolsep=2pt
\begin{array}{rcl}
\left[\,\beta_1(t),\ldots, \beta_{n_1}(t)\right]^T & = & a_1(t)=\left[\,\alpha_1^e(t),\ldots, \alpha_{n_{1}}^e(t)\right]^T, \\
\left[\,\beta_{n_{1}+1}(t),\ldots, \beta_{n_r}(t)\right]^T & = &  a_{21}(t).
\end{array}
\]
For $x_r(t)=[\beta_1(t),\ldots, \beta_{n_r}(t)]^T$ satisfying \eqref{eq:MQS3Dregevul}, we define a~storage function
\begin{equation*}
\mathcal{S}_r(x_r(t))=\int_{\varOmega} \vartheta(\xi,\|\nabla\times \sum_{i=1}^{n_r}\beta_i(t)\psi_i(\xi)\|^2)\, d\xi
\end{equation*}
with $\vartheta$ as in \eqref{eq:gamma}. Using \eqref{eq:xr}, we get $\mathcal{S}_r(x_r(t))=\mathcal{S}_{\,\FEM}(x_{FEM}(t))$ and, hence,
the passivity of the regularized DAE system \eqref{eq:MQS3Dreg} immediately follows from the proof of Theorem~\ref{th:passFEM}.
\end{proof}

\begin{remark}
	Since the transformation of the DAE control system \eqref{eq:MQS3Dreg} into the ODE control system \eqref{eq:ODE} does not change the input-output relation, the latter is also passive. 
\end{remark}

\section{Passivity-preserving POD model reduction}
\label{sec:pod}

For model reduction of the semidiscretized MQS system \eqref{eq:FEM-DAE}, \eqref{eq:FEM-DAE-matr}, we employ the POD method as proposed in~\cite{KerBS17}. Thought for general systems, this method does not necessarily preserves passivity, we will show that when applying to the structured MQS system \eqref{eq:FEM-DAE}, \eqref{eq:FEM-DAE-matr}, the preservation of passivity can be guaranteed. 

To proceed, we first transform the DAE system \eqref{eq:FEM-DAE} into the ODE system \eqref{eq:ODE} and then compute the reduced-order model by projecting
\begin{subequations}\label{eq:PODred}
\begin{align}
\tfrac{\rm d}{{\rm d}t} E\, {x_{\POD}} & = A(x_{\POD})\,x_{\POD} + B\,u, 
\quad x_{\POD}(0)=x_{\POD,0}, \label{eq:PODred1}\\
y_{\POD} & = C\,x_{\POD},\label{eq:PODred2}
\end{align}
\end{subequations}
with $x_{\POD}=\begin{bmatrix}\tilde{a}_1^T &\; \tilde{a}_{21}^T\end{bmatrix}^T$, $x_{\POD,0}=\tilde{U}^Tx_{\ODE,0}$, 
$A(x_{\POD})=\tilde{U}^T\mathcal{A}_{\ODE}(\tilde{U}x_{\POD})\,\tilde{U}$, \linebreak $E=\tilde{U}^T\mathcal{E}_{\ODE}\tilde{U}$, 
$B=\tilde{U}^T\mathcal{B}_{\ODE}$, 
$C=\mathcal{C}_{\ODE}\tilde{U}$, where the projection matrix is given by
\[
\tilde{U} = \begin{bmatrix} U_{a_1} &\; 0 \\ 0 &\; I_m\end{bmatrix}.
\]
Here, the columns of $U_{a_1}\in\mathbb{R}^{n_1 \times r}$ with $r\ll n_1$ form the POD basis of the snapshot matrix
\mbox{$X_{a_1}=[a_1(t_1), \ldots,a_1(t_q)]$}. The same reduced-order model can also be determined by first projecting the DAE system \eqref{eq:FEM-DAE} using the projection matrix
\begin{equation}
\begin{bmatrix} U_{a_1} &\; 0 &\; 0 &\; 0 \\ 0 &\; Z_2 &\; Y_2 &\; 0 \\ 0 &\; 0 &\; 0 &\; I_m\end{bmatrix}
\label{eq:proj}
\end{equation}
and then transforming the reduced DAE into the ODE form. To verify this, we substitute an~approximation
\begin{equation}
x_{\FEM}=\begin{bmatrix} a_1\\ a_2\\ \bm{i} \end{bmatrix} 
\approx \begin{bmatrix} U_{a_1} &\; 0 &\; 0 & 0 \\ 0 &\; Z_2 &\; Y_2 &\; 0 \\ 0 &\; 0 &\; 0 &\; I_m\end{bmatrix} \begin{bmatrix} \tilde{a}_1\\ \tilde{a}_{21}\\ \tilde{a}_{22}\\ \tilde{\bm{i}} \end{bmatrix}
\label{eq:a_appr}
\end{equation}
into \eqref{eq:FEM-DAE} and multiply this equation from the left with the transpose of the projection matrix in~\eqref{eq:proj}. This leads to the reduced-order DAE model
\begin{equation}
\arraycolsep=1pt
\begin{array}{rcrcl}
\tfrac{\rm d}{{\rm d}t}U_{a_1}^TM_{11}U_{a_1}^{}{\tilde{a}}_1 &\! = & \!\!   -U_{a_1}^TK_{11}(U_{a_1}\tilde{a}_1)U_{a_1}^{}\tilde{a}_1\! - U_{a_1}^TK_{12}Z_2\tilde{a}_{21}\! - U_{a_1}^TK_{12}Y_2\tilde{a}_{22} & + & U_{a_1}^TX_1\,\tilde{\bm{i}} \\
\qquad 0 \enskip &\! = & -Z_2^TK_{21}U_{a_1}\tilde{a}_1 -\, Z_2^TK_{22}Z_2^{}\tilde{a}_{21} - Z_2^TK_{22}Y_2^{}\tilde{a}_{22} & + & Z_2^TX_2\,\tilde{\bm{i}}  \\ 
\qquad 0 \enskip &\! = & -Y_2^TK_{21}U_{a_1}\tilde{a}_1\, -\, Y_2^TK_{22}Z_2^{}\tilde{a}_{21} - Y_2^TK_{22}Y_2^{}\tilde{a}_{22} &  & \\
\tfrac{\rm d}{{\rm d}t}\bigl(X_1^TU_{a_1}{\tilde{a}}_1\bigr. +\bigl. &\!X_2^T\!&\!\!\!\!\!\!\!
Z_2^{}{\tilde{a}}_{21}\bigr)=\qquad\qquad\qquad\qquad\qquad\qquad\qquad\qquad\quad  
\qquad& - & R\,\tilde{\bm{i}}  + u.
\end{array}
\label{eq:trDAE}
\end{equation}
Since $Y_2^TK_{22}Y_2^{}$ and $R$ are both nonsingular, we get from the third and fourth equations that
\begin{eqnarray}
\tilde{a}_{22} & = & -(Y_2^TK_{22}Y_2^{})^{-1} Y_2^T \bigl(K_{21}U_{a_1}\tilde{a}_1+K_{22}Z_2^{}\tilde{a}_{21}\bigr), \label{eq:a22} \\
y_{\POD} & = & \tilde{\bm{i}} = -R^{-1}\tfrac{\rm d}{{\rm d}t}\bigl(X_1^TU_{a_1}{\tilde{a}}_1+X_2^TZ_2^{}{\tilde{a}}_{21}\bigr)+R^{-1}u. \label{eq:tilde_y}
\end{eqnarray}
Substituting these vectors into the first and second equations in \eqref{eq:trDAE}, we obtain the reduced-order model \eqref{eq:PODred}. Inserting \eqref{eq:a22} into \eqref{eq:a_appr}, the discrete vector potential $a$ can be approximated as 
\begin{equation}
a=\begin{bmatrix}a_1\\a_2\end{bmatrix}\approx U \begin{bmatrix}\tilde{a}_1\\\tilde{a}_{21}\end{bmatrix}=U x_{\POD}
\label{eq:a_approx2}
\end{equation}
with the projection matrix 
\[
U  = \begin{bmatrix} U_{a_1} &\enskip 0 \\ -Y_2^{}(Y_2^TK_{22}Y_2^{})^{-1}Y_2^TK_{21}U_{a_1} &\enskip (I-Y_2^{}(Y_2^TK_{22}Y_2^{})^{-1}Y_2^TK_{22})Z_2\end{bmatrix}.
\]
Then the reduced matrices in the state equation \eqref{eq:PODred1} can be represented as 
\begin{align}
E & = U^T \mathcal{E}\, U= 
\begin{bmatrix} U_{a_1}^TM_{11}U_{a_1}^{} & 0\\ 0 & 0	\end{bmatrix}+
U^T\mathcal{B} R\, \mathcal{B}^TU,\notag\\[1mm]
A(x_{\POD}) & = -U^TK(Ux_{\POD})\,U,\qquad B = U^T\mathcal{B}. \label{eq:tildeA} 
\end{align}
Using these relations, we can show that the POD-reduced model \eqref{eq:PODred} preserves passivity.

\begin{theorem}\label{th:passPOD}
The POD-reduced model \eqref{eq:PODred} is passive.
\end{theorem}

\begin{proof}
The result can be proved analogously to Theorem~\ref{th:passFEMreg}. 
Let 
\[
\tilde{a}_1=[\tilde{\alpha}_1, \ldots, \tilde{\alpha}_r]^T, \qquad 
\tilde{a}_{21}=[\tilde{\alpha}_{r+1}, \ldots, \tilde{\alpha}_{r+m}]^T.
\]
Then using \eqref{eq:a_approx2}, we get an~approximation 
\[
\sum_{k=1}^{n_e} \alpha_k^e(t) \varphi_k^e(\xi) \approx 
\sum_{k=1}^{r+m} \tilde{\alpha}_k(t) \tilde{\psi}_k(\xi)
\]
with new basis functions
\begin{equation}
[\tilde{\psi}_1(\xi), \ldots, \tilde{\psi}_{r+m}(\xi)]=[\varphi_1^e(\xi), \ldots, \varphi_{n_e}^e(\xi)]\,U.
\label{eq:psi_basis}
\end{equation}
Define a~storage function for the reduced-order model \eqref{eq:PODred} as 
\[
\mathcal{S}_{\POD}(x_{\POD}(t)) = \int_{\varOmega} \vartheta\bigl(\xi,\|\nabla\times\sum_{j=1}^{r+m}\tilde{\alpha}_j(t)\tilde{\psi}_j(\xi)\|^2\bigr)\,{\rm d}\xi,
\]
where $x_{\POD}=[\tilde{\alpha}_1,\ldots,\tilde{\alpha}_{r+m}]^T$ and $\vartheta$ is as in \eqref{eq:gamma}. 
Using \eqref{eq:psi_basis}, \eqref{eq:tildeA} and \eqref{eq:PODred}, we compute the time derivative 
\begin{align*}
&\displaystyle{\tfrac{\rm d}{{\rm d}t}{\mathcal{S}_{\POD}}(x_{\POD}(t))} =  
\displaystyle{ \int_\varOmega\tfrac{\partial}{\partial\zeta}
\vartheta\bigl(\xi,\|\nabla\times\sum_{j=1}^{r+m}\tilde{\alpha}_j(t)\tilde{\psi}_j(\xi)\|^2\bigr)
\tfrac{\partial}{\partial t}\|\nabla\times\sum_{j=1}^{r+m}\tilde{\alpha}_j(t)\tilde{\psi}_j(\xi)\|^2 \,{\rm d}\xi} \\
&\quad = \displaystyle{\!\int_\varOmega \nu\bigl(\xi,\|\nabla\!\times\!\sum_{j=1}^{r+m}\tilde{\alpha}_j(t)\tilde{\psi}_j(\xi)\|\bigr)
\Bigl(\nabla\!\times\!\sum_{k=1}^{r+m}\tfrac{\rm d}{{\rm d}t}{\tilde{\alpha}}_k(t)\tilde{\psi}_k(\xi)\Bigr)\cdot
\Bigl(\nabla\!\times\!\sum_{l=1}^{r+m}\tilde{\alpha}_l(t)\tilde{\psi}_l(\xi)\Bigr){\rm d}\xi} \\
&\quad = \displaystyle{\!\sum_{k=1}^{r+m}\!\tfrac{\rm d}{{\rm d}t}{\tilde{\alpha}}_k(t)\!\sum_{l=1}^{r+m}\!\tilde{\alpha}_l(t)\!
\int_\varOmega \nu\bigl(\xi,\|\nabla\!\times\!\sum_{j=1}^{r+m}\tilde{\alpha}_j(t)\tilde{\psi}_j(\xi)\|\bigr)
\Bigl(\nabla\!\times\!\tilde{\psi}_k(\xi)\Bigr)\cdot
\Bigl(\nabla\!\times\!\tilde{\psi}_l(\xi)\Bigr){\rm d}\xi} \\[3mm]
&\quad = \bigl(\tfrac{\rm d}{{\rm d}t}{x_{\POD}}(t)\bigr)^T U^T K(Ux_{\POD}(t)) Ux_{\POD}(t) \\[3mm]
&\quad = -\bigl(\tfrac{\rm d}{{\rm d}t}{x_{\POD}}(t)\bigr)^T A(x_{\POD}(t)) x_{\POD}(t) \\[3mm]
&\quad = \displaystyle{-\bigl(\tfrac{\rm d}{{\rm d}t}{x_{\POD}}(t)\bigr)^T  \tfrac{\rm d}{{\rm d}t}E\,{x_{\POD}}(t) + \bigl(\tfrac{\rm d}{{\rm d}t}{x_{\POD}}(t)\bigr)^TB\,u(t)}\\[3mm]
&\quad = \displaystyle{-\bigl(\tfrac{\rm d}{{\rm d}t}{\tilde{a}}_1(t)\bigr)^TU_{a_1}^TM_{11}U_{a_1}\tfrac{\rm d}{{\rm d}t}{\tilde{a}}_1(t)
-\bigl(\tfrac{\rm d}{{\rm d}t}{x_{\POD}}(t)\bigr)^TBRB^T\tfrac{\rm d}{{\rm d}t}{x_{\POD}}(t)}\\[3mm]
&\quad\quad+\displaystyle{\bigl(\tfrac{\rm d}{{\rm d}t}{x_{\POD}}(t)\bigr)^TBu(t).}
\end{align*}
Taking into account the positive definiteness of the matrices $M_{11}$ and $R$ as well as the relation 
$y_{\POD}(t)=-\tfrac{\rm d}{{\rm d}t}B^Tx_{\POD}(t)+R^{-1}u(t)$, which follows from \eqref{eq:tilde_y}, we obtain that
\[
\arraycolsep=2pt
\begin{array}{rcl}
\tfrac{\rm d}{{\rm d}t}{\mathcal{S}_{\POD}}(x_{\POD}(t)) & \leq &
-\bigl(\tfrac{\rm d}{{\rm d}t}B^Tx_{\POD}(t)\bigr)^TR\bigr(\tfrac{\rm d}{{\rm d}t}B^Tx_{\POD}(t)-R^{-1}u(t)\bigr)\\[2mm]
& = & -y_{\POD}^T(t)R\, y_{\POD}^{}(t) +u^T(t)\,y_{\POD}^{}(t)\, \leq  
u^T(t)\,y_{\POD}^{}(t).
\end{array}
\] 
Integrating this inequality on $[t_0,t_1]\subseteq [0,T]$, we get the dissipation inequality
\[
{\mathcal{S}_{\POD}}(x_{\POD}(t_1))-{\mathcal{S}_{\POD}}(x_{\POD}(t_0)) \leq \int_{t_0}^{t_1} u^T(\tau)\,y_{\POD}^{}(\tau)\, {\rm d}\tau,
\]
which implies the passivity of the reduced-order model \eqref{eq:PODred}. 
\end{proof}

\section{Enforcing passivity for the POD-DEIM-reduced model}
\label{sec:deim}

The simulation of the POD-reduced model \eqref{eq:PODred} can significantly be accelerated by using the DEIM \cite{ChaS10} for fast evaluation of the non\-li\-nea\-ri\-ty 
$A(x_{\POD})x_{\POD}$. Taking into account the structure of $\nu$ on the conducting and non-conducting subdomains, the block $K_{11}(a_1)$ of $K(a)$ in \eqref{eq:blockmatr} can be written as 
$K_{11}(a_1)=K_{11,l}+K_{11,n}(a_1)$ with a constant matrix $K_{11,l}$ and a~matrix-valued nonlinear function $K_{11,n}(a_1)$. Then the nonlinearity of  \eqref{eq:PODred} takes the form
\begin{equation}
A(x_{\POD})x_{\POD}=A_l\,x_{\POD} +\begin{bmatrix}U_{a_1}^T f_1(U_{a_1}^{}\tilde{a}_1)\\ 0\end{bmatrix},
\label{eq:Atildex}
\end{equation}
where $f_1(U_{a_1}\tilde{a}_1)=-K_{11,n}(U_{a_1}\tilde{a}_1)U_{a_1}\tilde{a}_1$ and $A_l=-U^TK_l\,U$ with
\[
K_l= \begin{bmatrix} K_{11,l} &\; K_{12} \\ K_{21} &\; K_{22} \end{bmatrix}.
\] 
Applying the DEIM to $f_1(U_{a_1}\tilde{a}_1)$ as described in \cite{KerBS17}, we obtain a POD-DEIM model 
\begin{equation}
\arraycolsep=2pt
\begin{array}{rcl}
\tfrac{\rm d}{{\rm d}t}E\, {x_{\DEIM}} & = & F(x_{\DEIM})\,x_{\DEIM} + B\,u,
\quad x_{\DEIM}(0)=x_{\DEIM,0}, \\
y_{\DEIM} & = & C\,x_{\DEIM},
\end{array}
\label{eq:POD-DEIMred}
\end{equation}
where $x_{\DEIM}=[\hat{a}_1^T, \hat{a}_{21}^T]^T$, $x_{\DEIM,0}=x_{\POD,0}$, and
\[
F(x_{\DEIM})x_{\DEIM}=A_l\,x_{\DEIM}+\begin{bmatrix}\hat{f}_1(\hat{a}_1)\\0\end{bmatrix},
\]
with the DEIM approximation $\hat{f}_1(\hat{a}_1)\!=\!U_{a_1}^TU_{\!f_1}(S_{\mathcal{K}}^TU_{\!f_1})^{-1}S_{\mathcal{K}}^Tf_1(U_{a_1}\hat{a}_1)$.
Here, \mbox{$U_{\!f_1}\!\in\!\mathbb{R}^{n_1\times \ell}$} is the DEIM basis matrix obtained from the snapshot matrix 
\[
X_{f_1}=\big[f_1(a_1(t_1)),\ldots,f_1(a_1(t_q))\big],
\] 
and \mbox{$S_{\mathcal{K}}\in\mathbb{R}^{n_1\times \ell}$} is the selector matrix associated with a~DEIM index set $\mathcal{K}$ determined by a greedy procedure, see \cite{KerBS17} for detail. One can see that only a few selected components of the nonlinear function $f_1$ need to be evaluated at each time integration step. 

It should be noted, however, that the DEIM does not preserve the underlying symmetric system structure, which was used for the construction of the storage functions above, and, as a~consequence, we can no longer guarantee the preservation of passivity. To remedy this problem, we propose to enforce the io-passivity of the POD-DEIM-reduced model \eqref{eq:POD-DEIMred} by a~small perturbation of the output. 

We aim to find a~scalar function $\delta:[0,T]\to\mathbb{R}$ such that the perturbed control system
\begin{equation}
\arraycolsep=2pt
\begin{array}{rcl}
\tfrac{\rm d}{{\rm d}t}E\,{x_{\DEIM}} &=&  F(x_{\DEIM})\,x_{\DEIM} + B\, u,
\quad x_{\DEIM}(0)=x_{\DEIM,0}, \\
y_\delta & = & C \,x_{\DEIM} + \delta\, u
\end{array}
\label{eq:delta-sys}
\end{equation}
is io-passive and the output error $\|y_{\DEIM}-y_\delta\|$ is small. For this purpose, we consider the POD-reduced system~\eqref{eq:PODred} with the state $x_{\POD}$ and the POD-DEIM-reduced system~\eqref{eq:POD-DEIMred} with the state
$x_{\DEIM}$. Then for the DEIM state error $\varepsilon=x_{\POD} -x_{\DEIM}$, we obtain for all $t\in[0,T]$ that
\begin{align}
  \int_0^{t} &u^T(\tau) \, y_\delta^{}(\tau)\, {\rm d}\tau  =  
  \displaystyle{\int_0^{t} u^T(\tau) \bigl(C x_{\DEIM}(\tau) +\delta(\tau) u(\tau)\bigr)\, {\rm d}\tau} \nonumber\\
  &= \displaystyle{\int_0^{t} u^T(\tau) \bigl(C x_{\POD}(\tau)-C \varepsilon(\tau)+
	\delta(\tau) u(\tau)\bigr) \, {\rm d}\tau}\nonumber\\
  &\geq \displaystyle{\int_0^{t} \!u^T(\tau)\, y_{\POD}(\tau)\, {\rm d}\tau + 
  \int_0^{t} \!\delta(\tau) \|u(\tau)\|^2 {\rm d}\tau- \int_0^{t} \!\|C\|\|\varepsilon(\tau)\|\,
	\|u(\tau)\|\,{\rm d}\tau}. 
  \label{eq:integrals}
\end{align}
Since the POD-reduced system \eqref{eq:PODred} is passive and $\mathcal{S}_{\rm POD}(0)=0$, it is also io-passive by Remark~\ref{rem:iopass}. Therefore, the first integral in \eqref{eq:integrals} is non-negative. 
By choosing 
\begin{equation}
\delta(t) \geq \left\{\begin{array}{ll} 
\displaystyle{\frac{\|C\|\|\varepsilon(t)\|}{\|u(t)\|}}, & \quad\text{if}\enskip \|u(t)\|\neq 0,  \\ 
 0,                                                      & \quad\text{if}\enskip \|u(t)\|= 0,
                      \end{array}
\right.
\label{eq:delta_est}
\end{equation}
for all $t\in[0,T]$, we obtain that $\delta(t) \|u(t)\|^2 - \|C\|\,\|\varepsilon(t)\| \|u(t)\|\geq 0$ for all $t\in[0,T]$ and, hence, 
\[
\int_0^{t} u^T(\tau)\,y_{\delta}(\tau) \,{\rm d}\tau \geq 0.
\]
This implies that the perturbed  system \eqref{eq:delta-sys} is io-passive. 

The computation of $\delta(t)$ relies on the DEIM state error $\varepsilon(t)$ which is not readily available. Instead, we derive a~computable bound $\|\varepsilon(t)\|\leq \theta(t)$ which allows us to determine $\delta(t)$ as
\begin{equation}
\delta(t) = \left\{\begin{array}{ll} 
\displaystyle{\frac{\|C\|\,\theta(t)}{\|u(t)\|}}, & \quad \text{if}\enskip \|u(t)\|\neq 0,  \\ 
                      0, & \quad\text{if}\enskip \|u(t)\|= 0.
                      \end{array}
\right.
\label{eq:delta}
\end{equation}
It, obviously, satisfies \eqref{eq:delta_est} and 
the perturbation $\delta(t)u(t)$ remains bounded if~$\theta(t)$ is bounded.
For $\delta(t)$ as in \eqref{eq:delta}, the output error for systems \eqref{eq:POD-DEIMred} and \eqref{eq:delta-sys} is given by
\[
\|y_{\DEIM}(t)-y_\delta(t)\| = \|\delta(t)u(t)\|= \|C\|\,\theta(t).
\]
Moreover, taking into account that
\[
\|y_{\POD}(t)-y_{\DEIM}(t)\| =\|Cx_{\POD}(t)-Cx_{\DEIM}(t)\|\leq \|C\|\,\theta(t),
\]
we estimate the output error for systems \eqref{eq:PODred} and \eqref{eq:delta-sys} as
\[
\|y_{\POD}(t)-y_\delta(t)\| \leq \|y_{\POD}(t)-y_{\DEIM}(t)\|+\|y_{\DEIM}(t)-y_\delta(t)\| \leq 2 \|C\|\,\theta(t).
\]

In order to derive a bound on $\|\varepsilon(t)\|$, we make use of a~{\em logarithmic Lipschitz constant} $L_2[f]$ for a~nonlinear function $f(z)$ defined as
\begin{equation}
L_2[f] = \sup_{z\neq w} \frac{\big(f(z)-f(w)\big)^T\big(z-w\big)}{\|z-w\|^2}.
\label{eq:llLc}
\end{equation}
Note that for a~linear function $f(z)=Az$ with a~constant matrix $A$, we have
\[
L_2[f]  
= \sup_{z\neq w}\frac{(z-w)^TA(z-w)}{\|z-w\|^2} 
= \sup_{z\neq 0} \frac{z^TAz}{\|z\|^2} 
= \lambda_{\max}\left(\frac{A+A^T}{2}\right)=:L_2[A],
\]
where $\lambda_{\max}(\cdot)$ denotes the largest eigenvalue of the corresponding matrix. The value 
$L_2[A]$ is also known as the logarithmic norm of $A$, see \cite{Dahl1958,Soed2006}, although it is not a~norm in the usual sense. For example, $L_2[A]<0$ for a~symmetric, negative definite matrix $A$.

The following theorem provides a bound on the DEIM state error $\varepsilon$.

\begin{theorem}\label{th:state_err}
Consider the POD-reduced system \eqref{eq:PODred} with the state $x_{\POD}$ and the POD-DEIM-reduced system \eqref{eq:POD-DEIMred} with the state
$x_{\DEIM}$. Then the state error \mbox{$\varepsilon=x_{\POD} -x_{\DEIM}$} can be estimated as 
\begin{equation}
  \|\varepsilon(t)\| \leq \frac{\Delta_{\rm DEIM}}{L_2[f]} 
	\Bigl(e^{\textstyle{\frac{L_2[f]}{\lambda_{\min}(E)}}t}-1\Bigr),
	\label{eq:est_eps}
\end{equation}
where $L_2[f]$ is the~logarithmic Lipschitz constant of $f(x_{\POD})=A(x_{\POD})x_{\POD}$ defined in \eqref{eq:llLc}, $\lambda_{\min}(E)$ is the smallest eigenvalue of $E$, and 
\begin{equation}
\Delta_{\rm DEIM}=\big\|(S_{\mathcal{K}}^TU_{f_1})^{-1}\big\| \sum_{j=\ell+1}^q \Bigl(\varsigma_j^2(X_{f_1})\Bigr)^{1/2}
\label{eq:beta}
\end{equation}
with the truncated singular values $\varsigma_j(X_{f_1})$ of the DEIM snapshot matrix $X_{f_1}$.
\end{theorem}

\begin{proof}
Subtracting the POD-DEIM-reduced system \eqref{eq:POD-DEIMred} from the POD-reduced system~\eqref{eq:PODred}, we obtain the following differential equation for the error
\begin{equation}\label{eq:odevarepsilon}
\begin{array}{rclrl}
 \tfrac{\rm d}{{\rm d}t}E\,{\varepsilon}(t) = A_l\,\varepsilon(t)+\begin{bmatrix}
                         U_{a_1}^T f_1(U_{a_1}\tilde{a}_1(t))-\hat{f}_1(\hat{a}_1(t))\\0
                        \end{bmatrix}.
\end{array}
\end{equation}
We consider now a~weighted vector norm
$\|w\|_{E}=\sqrt{w^TEw}$ for $w\in\mathbb{R}^{r+m}$. It is well defined since $E$ is symmetric and positive definite. 
Clearly, as this holds for any two norms in $\mathbb{R}^n$, this norm is equivalent to the Euclidean norm $\|\cdot\|$. More precisely, we have the inequalities
\begin{equation}
\sqrt{\lambda_{\min}(E)}\|w\|\leq \|w\|_{E} \leq \sqrt{\lambda_{\max}(E)}\|w\|.
\label{eq:normeq}
\end{equation}
Using \eqref{eq:odevarepsilon} and the definition of $L_2[f]$, we can estimate
\begin{align*}
  \|\varepsilon(t)\|_{E} & \displaystyle{\tfrac{\rm d}{{\rm d}t}} \|\varepsilon(t)\|_{E} 
  = \varepsilon^T(t)\tfrac{\rm d}{{\rm d}t}E\,{\varepsilon}(t) \\
  &= \varepsilon^T(t) \Bigg(A_l\,\varepsilon(t)+\begin{bmatrix}
                         U_{a_1}^T f_1(U_{a_1}\tilde{a}_1(t))-\hat{f}_1(\hat{a}_1(t))\\0
                        \end{bmatrix} \Bigg) \\
  & =  \varepsilon^T(t) \Bigg( A_l\,\varepsilon(t)+\begin{bmatrix}
                         U_{a_1}^T f_1(U_{a_1}\tilde{a}_1(t))-U_{a_1}^T f_1(U_{a_1}\hat{a}_1(t))\\0
                        \end{bmatrix} \Bigg)\\
  & \quad + \varepsilon^T(t) \begin{bmatrix}U_{a_1}^T f_1(U_{a_1}\hat{a}_1(t))-\hat{f}_1(\hat{a}_1(t))\\0\end{bmatrix} \\
  & = \big(x_{\POD}(t)-x_{\DEIM}(t)\big)^T\big(A(x_{\POD}(t))x_{\POD}(t)- 
  A(x_{\DEIM}(t))x_{\DEIM}(t)\big) \\
  & \quad + \big(\tilde{a}_1(t)-\hat{a}_1(t)\big)^T\big(U_{a_1}^T f_1(U_{a_1}\hat{a}_1(t))-\hat{f}_1(\hat{a}_1(t))\big)\\
	&\leq L_2[f]\|\varepsilon(t)\|^2 +\Delta_{\rm DEIM}\|\varepsilon(t)\|.
\end{align*}
In the last inequality, we used the estimate for the DEIM error
\begin{align*}
\big\|U_{a_1}^T f_1(U_{a_1}\hat{a}_1(t))-\hat{f}_1(\hat{a}_1(t))\big\| 
    & \leq \big\|(S_{\mathcal{K}}^TU_{f_1})^{-1}\big\| 
    \big\|(I-U_{f_1}^{}U_{f_1}^T)f_1(U_{a_1}\hat{a}_1(t))\big\|\\
& \lesssim \displaystyle{\big\|(S_{\mathcal{K}}^TU_{f_1})^{-1}\big\|
\Bigl(\sum_{j=\ell+1}^q\varsigma_j^2(\mathcal{X}_{f_1})\Bigr)^{1/2}=:\Delta_{\rm DEIM}}
\end{align*}
derived in \cite{ChaS10}. Taking into account \eqref{eq:normeq}, we obtain that
\begin{equation}
 \displaystyle{\textstyle{\frac{\rm d}{{\rm d}t}}} \|\varepsilon(t)\|_{E} \leq 
\displaystyle{L_2[f]\frac{\|\varepsilon(t)\|^2\,}{\|\varepsilon(t)\|_{E}}
+\Delta_{\rm DEIM}\frac{\|\varepsilon(t)\|\;\;}{\|\varepsilon(t)\|_{E}} \leq
\frac{L_2[f]}{\lambda_{\min}(E)} \|\varepsilon(t)\|_{E}
	+\frac{\Delta_{\rm DEIM}}{\sqrt{\lambda_{\min}(E)}}}.
\label{eq:est_doteps}
\end{equation}
Further, Gronwall's inequality \cite{Pach98} yields
\[
\|\varepsilon(t)\|_{E} \leq   \displaystyle{\int_0^t \frac{\Delta_{\rm DEIM}}{\sqrt{\lambda_{\min}(E)}} 
e^{\textstyle{\int_s^t} \frac{L_2[f]}{\lambda_{\min}(E)}\,d\tau}  \,ds}
=\displaystyle{\frac{\Delta_{\rm DEIM} \sqrt{\lambda_{\min}(E)}}{L_2[f]} \Bigl(
e^{\textstyle{\frac{L_2[f]}{\lambda_{\min}(E)}}t}-1\Bigr)}.
\]
Finally, we use the norm equivalence \eqref{eq:normeq} once again and obtain \eqref{eq:est_eps}. 
\end{proof}

Next, we present a~computable bound on the logarithmic Lipschitz constant $L_2[f]$.

\begin{theorem} \label{th:L2est}
The logarithmic Lipschitz constant $L_2[f]$ for $f(x_{\POD})=A(x_{\POD})x_{\POD}$ can be estimated as $L_2[f] \leq \min(\mu_1,\mu_2)$ with
\begin{equation}
\arraycolsep=2pt
\begin{array}{rcl}
\mu_1 & = & -m_{\nu} \lambda_{\max}(U^TC_d^TM_fC_d^{}U), \\
\mu_2 & = & \lambda_{\max}(A_l) -m_{\nu_C} \lambda_{\max}(U_{a_1}^TC_1^TM_{f,1}C_1^{}U_{a_1}^{}),
\end{array}
\label{eq:mu12}
\end{equation}
where the matrices $M_f$ and $M_{f,1}$ have the entrees 
\[
\displaystyle{
(M_f)_{kl} = \int_{\varOmega} \varphi_l^f\cdot\varphi_k^f\,{\rm d}\xi, \qquad
(M_{f,1})_{kl} = \int_{\varOmega_{1}} \varphi_l^f\cdot\varphi_k^f\,{\rm d}\xi}, \qquad 
k,l=1,\ldots,n_f.
\]
\end{theorem}

\begin{proof}
Consider the operator $\cAl_{11}$ in \eqref{eq:A11}. By \cite[Lem.~3]{ChiRS23}, 
for all \mbox{$\psi,\phi\in X_0(\curl,\varOmega,\varOmega_C)$}, it holds that
\[
\bigl\langle \cAl_{11}(\psi)-\cAl_{11}(\phi),\psi-\phi\bigr\rangle
\geq m_\nu \, \|\nabla\times (\psi-\phi)\|_{L^2(\varOmega;\mathbb{R}^3)}^2,
\]
where $m_\nu$ is the monotonicity constant of $\nu(\cdot,\zeta)\zeta$.
Using this inequality and \eqref{eq:K}, we obtain for $f(x_{\POD})=A(x_{\POD})x_{\POD}=-U^TK(Ux_{\POD})\,Ux_{\POD}$ that
\begin{align}
\big(f(x_{\POD})\bigr.& - \bigl. f(x_{\DEIM})\big)^T\big(x_{\POD}-x_{\DEIM}\big) \nonumber\\
& =  -\big(K(Ux_{\POD})Ux_{\POD}-K(Ux_{\DEIM})Ux_{\DEIM}\big)^T\big(Ux_{\POD}-Ux_{\DEIM}\big)\nonumber\\[1mm]
& = -\Bigl\langle\cAl_{11}\bigl(\sum_{j=1}^{r+m} \tilde{\alpha}_j\tilde{\psi}_j\bigr)-
	\cAl_{11}\bigl(\sum_{j=1}^{r+m} \hat{\alpha}_j\tilde{\psi}_j\bigr), 
	\sum_{j=1}^{r+m} (\tilde{\alpha}_j\!-\!\hat{\alpha}_j)\tilde{\psi}_j\Bigr\rangle \nonumber\\
& \leq -m_\nu \Bigl\|\nabla\times \sum_{j=1}^{r+m} (\tilde{\alpha}_j-\hat{\alpha}_j)\tilde{\psi}_j\Bigr\|_{L^2(\varOmega;\mathbb{R}^3)}^2 \label{eq:estLip} \\
& = -m_\nu \int_{\varOmega}\Bigl( \nabla\times\sum_{j=1}^{r+m} (\tilde{\alpha}_j-\hat{\alpha}_j)\tilde{\psi}_j\Bigr)\cdot
	\Bigl( \nabla\times\sum_{j=1}^{r+m} (\tilde{\alpha}_j-\hat{\alpha}_j)\tilde{\psi}_j\Bigr)\, d\xi \nonumber   
\\[1mm]
& = -m_\nu \bigl(x_{\POD}-x_{\DEIM}\bigr)^TU^TC_d^TM_f\,C_d^{}\,U\bigl(x_{\POD}-x_{\DEIM}\bigr)\nonumber
\end{align}
for all $x_{\POD}=[\tilde{\alpha}_1,\ldots,\tilde{\alpha}_{r+m}]^T$ and $x_{\DEIM}=[\hat{\alpha}_1,\ldots,\hat{\alpha}_{r+m}]^T$.
Therefore,
\begin{align*}
L_2[f] & = \sup_{x_{\sPOD}\neq x_{\sDEIM}} \frac{\big(f(x_{\POD})-f(x_{\DEIM})\big)^T\big(x_{\POD}-x_{\DEIM} \big)}{\|x_{\POD}-x_{\DEIM}\|^2} \\
& \leq -m_\nu\, \sup_{x_{\sPOD}\neq x_{\sDEIM}} \frac{\bigr(x_{\POD}-x_{\DEIM})\big)^T U^TC_d^TM_f\,C_d^{}\,U \big(x_{\POD}-x_{\DEIM}\big)}{\|x_{\POD}-x_{\DEIM}\|^2}\\
& = -m_{\nu}\,\lambda_{\max}\bigl(U^TC_d^TM_f\,C_d^{}\,U\bigr)=:\mu_1.
\end{align*}  

Alternatively, we can consider the additive decomposition \eqref{eq:Atildex}. Then we have
\begin{align*}
 L_2[f] 
  & \leq L_2[A_l] + \sup_{\tilde{a}_1\neq\hat{a}_1} 
 	\frac{\big(f_1(U_{a_1}\tilde{a}_1)-f_1(U_{a_1}\hat{a}_1)\big)^T\big(U_{a_1}\tilde{a}_1-U\hat{a}_1\big)}{\|\tilde{a}_1-\hat{a}_1\|^2}.
\end{align*}  

Since $A_l$ is symmetric, it holds $L_2[A_l]=\lambda_{\max}(A_l)$. Furthermore, similarly to \eqref{eq:estLip}, we can show that
\[
\big(f_1(U_{a_1}\tilde{a}_1)\!-\!f_1(U_{a_1}\hat{a}_1)\big)^T\big(U_{a_1}\tilde{a}_1\!-\!U\hat{a}_1\big)
\leq \! -m_{\nu_C} (\tilde{a}_1\!-\!\hat{a}_1)^TU_{a_1}^TC_1^TM_{f,1}C_1^{}U_{a_1}^{}(\tilde{a}_1\!-\!\hat{a}_1),
\]
where $m_{\nu_C}$ is the monotonicity constant of $\nu_C(\zeta)\zeta$. Then
\begin{align*}
L_2[f] & \leq \lambda_{\max}(A_l) -m_{\nu_C}\,\displaystyle{\sup_{\tilde{a}_1\neq\hat{a}_1} 
	\frac{(\tilde{a}_1-\hat{a}_1)^T  U_{a_1}^TC_1^TM_{f,1}C_1^{}U_{a_1}^{}(\tilde{a}_1-\hat{a}_1)
	}{\|\tilde{a}_1-\hat{a}_1\|^2} }\\
& =  \lambda_{\max}(A_l) -m_{\nu_C}\,\lambda_{\max}\bigl(U_{a_1}^TC_1^TM_{f,1}C_1^{}U_{a_1}^{}\bigr)=:\mu_2.
\end{align*}  
Thus, $L_2[f]\leq \min\{\mu_1,\mu_2 \}$. 
\end{proof}

Note that since the matrices $M_f$ and $M_{f,1}$ are symmetric, positive definite and the matrix $A_l$ is symmetric, negative semidefinite, we have $\mu=\min\{\mu_1,\mu_2\}<0$. 

\begin{remark}\label{rem:logLip2D}
In the 2D case, the logarithmic Lipschitz constant is estimated as $L_2[f] \leq \min(\mu_1,\mu_2)$ with
\begin{align*}
	\mu_1& = -m_{\nu}\lambda_{\max}(U^TK_LU), \\
	\mu_2& = \lambda_{\max}(A_{l})-m_{\nu_C}\lambda_{\max}(U_{a_1}^TK_{L,1}^{} U_{a_1}^{}),
\end{align*}
	where 
\[
	U = \begin{bmatrix} U_{a_1} &\enskip 0 \\ -Y^{}(Y^TK_{22}Y^{})^{-1}Y^TK_{21}U_{a_1} &\enskip
	(I-Y^{}(Y^TK_{22}Y^{})^{-1}Y^TK_{22})Z\end{bmatrix}
\]
	with $\mbox{\rm im}(Y)=\mbox{\rm ker}(X_2^T)$, $Y^TY=I$  and $Z=X_2^{}(X_2^TX_2^{})^{-1/2}$, and the stiffness matrices $K_L$ and $K_{L,1}$ have the entrees 
\[
	\arraycolsep=2pt
	\begin{array}{rcll}
	(K_L)_{kl} & = & \displaystyle{
		\int_{\varOmega} \nabla\varphi_l\cdot\nabla\varphi_k\,d\xi,} & \qquad k,l=1,\ldots,n_a, \\[3mm]
	(K_{L,1})_{kl} & = & \displaystyle{\int_{\varOmega_C} \nabla\varphi_l\cdot\nabla\varphi_k\,d\xi,} & \qquad k,l=1,\ldots,n_1,
	\end{array}
\]
	with the Lagrange basis functions $\varphi_k$. 
\end{remark}

Combining estimate \eqref{eq:est_doteps} with Theorem~\ref{th:L2est}, we obtain the following estimate for the DEIM state error $\varepsilon$.

\begin{theorem}\label{th:state_err2}
	Consider the POD-reduced system \eqref{eq:PODred} with the state $x_{\POD}$ and the POD-DEIM-reduced system \eqref{eq:POD-DEIMred} with the state $x_{\DEIM}$. Then the DEIM state error \linebreak \mbox{$\varepsilon=x_{\POD} -x_{\DEIM}$} can be estimated as 
	\begin{equation}
	\|\varepsilon(t)\| \leq \frac{\Delta_{\rm DEIM}}{\mu}\Bigl(e^{\textstyle{\frac{\mu}{\lambda_{\min}(E)}}t}-1\Bigr)=:\theta(t),
	\label{eq:epsilon}
	\end{equation}
	where $\mu=\min\{\mu_1,\mu_2\}$ with $\mu_1$ and $\mu_2$ as in \eqref{eq:mu12} and $\Delta_{\rm DEIM}$ is given in \eqref{eq:beta}. 
\end{theorem}

\begin{proof}
	Using Theorem~\ref{th:L2est},  we obtain from \eqref{eq:est_doteps} that
\[
 \displaystyle{\textstyle{\frac{\rm d}{{\rm d}t}}} \|\varepsilon(t)\|_{E} \leq 
\displaystyle{\frac{\mu}{\lambda_{\min}(E)} \|\varepsilon(t)\|_{E}
	+\frac{\Delta_{\rm DEIM}}{\sqrt{\lambda_{\min}(E)}}}.
\]
Then the DEIM state error  \eqref{eq:epsilon} follows from the norm equivalence \eqref{eq:normeq} and Gronwall's inequality
similarly to the proof of Theorem~\ref{th:state_err}. 
\end{proof}

In order to compute the state error bound $\theta(t)$ in \eqref{eq:epsilon}, we have to calculate the corresponding eigenvalues
of the reduced matrices $E, A_l, U^TC_d^TM_fC_d^{}U\in\mathbb{R}^{(r+m)\times(r+m)}$ and $U_{a_1}^TC_1^TM_{f,1}C_1^{}U_{a_1}^{}\in\mathbb{R}^{r\times r}$. Thus, the computational cost does not depend on the dimension $n_e+m$ of the original problem \eqref{eq:FEM-DAE}.

\section{Numerical results}
\label{sec:num}

In this section, we present some results of nume\-ri\-cal experiments for a single-phase 2D transformer model with an~iron core and two coils of wire in an air domain as described in \cite{KerBS17, Schoeps11}. For the mesh generation and the FEM discretization, we used the software package \mbox{FEniCS}\footnote{http://fenicsproject.org}. 
The time integration of the full models is done by the sparse DAE solver PyDAESI, 
whereas the reduced-order dense systems are solved by the implicit differential-algebraic (IDA) solver from the simulation package Assimulo\footnote{http://www.jmodelica.org/assimulo}.
Both solvers are based on the backward differentiation formula methods. We use them according to the sparse or dense structure of the problem. The computations were performed on a computer with an Intel(R) Core(TM) i7-3720QM processor with 2.60GHz.

The winding function $\chi$ has two components resulting in input dimension \mbox{$m=2$}.
The FEM discretization with linear Lagrange elements on a uniform triangular mesh yields  
a~semidiscretized MQS system \eqref{eq:FEM-DAE}, \eqref{eq:FEM-DAE-matr} with a~positive definite matrix \mbox{$K(a)\in\mathbb{R}^{n_a\times n_a}$} with $n_a=51543=n_1+n_2$, where $n_1=19688$ and $n_2=31855$ are the numbers of conductive and non-conductive state components, respectively. 

The snapshots were collected by solving this system with the training input $u(t)$ and the reduced models were tested by using the input $u_{\rm test}(t)$ given by
\begin{equation} \label{eq:NonlinExinput}
u(t)=\begin{bmatrix}
45.5 \cdot 10^3 \sin(900\pi t)\\
77 \cdot 10^3 \sin(1700\pi t)
\end{bmatrix},\quad
u_{\rm test}(t)=\begin{bmatrix}
46.5 \cdot 10^3 \sin(1010\pi t)\\
78 \cdot 10^3 \sin(1900\pi t)
\end{bmatrix}.
\end{equation}
The reduced dimensions were chosen as $r=35$ for the POD method and $\ell=9$ for the DEIM method with the relations \vspace*{-1.5mm}
\begin{equation*}
\frac{\varsigma_{36}(X_{a_1})}{\varsigma_1(X_{a_1})}=1.14\cdot 10^{-7},\qquad \frac{\varsigma_{10}(X_{f_1})}{\varsigma_1(X_{f_1})}=1.40\cdot 10^{-3} \vspace*{-1.5mm}
\end{equation*} 
for the singular values $\varsigma_j(X_{a_1})$ and $\varsigma_j(X_{f_1})$ of the POD and DEIM snapshot matrices $X_{a_1}$ and $X_{f_1}$, respectively. The resulting reduced models have dimension $r+m=37$. 

\begin{table}[t]
\caption{Constants for the state error bound $\theta$}
	\label{tab:DEIMerrorconst}
	\centering
    \begin{tabular}{lll}
		\toprule
		$\Delta_{\rm DEIM}=4.1467$  &	$\qquad\lambda_{\min}(E)=0.1969$ & $\qquad\|C\|=1.591\cdot 10^{-3}$\\
		$\mu_1=-7.7916$ &	$\qquad m_\nu=1,\enskip m_{\nu_C}=396.2$  & $\qquad \lambda_{\max}(U^TK_L\,U)=7.7916$\\
		$\mu_2=-2755$   & 
		$\qquad \lambda_{\max}(A_l)=-1.4307\cdot10^{-4}$ 
		&	$\qquad \lambda_{\max}(U_{a_1}^TK_{L,1}\, U_{a_1}^{})=6.9569$\\
		\botrule
	\end{tabular}
\end{table}

\begin{figure}[t] %1
 \centering
 \includegraphics{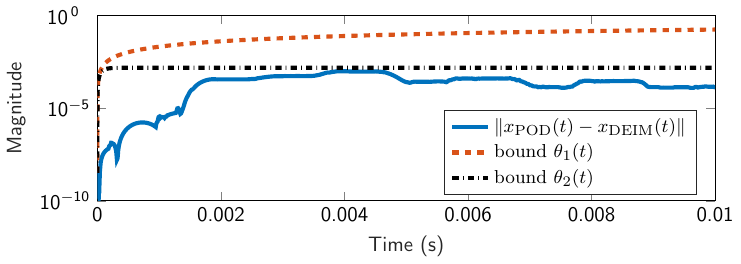} 
 \caption{Absolute error in the state and error bounds}
 \label{fig:NumExNonstateerror}
\end{figure}

\begin{figure}[t] %2
\centering
\includegraphics{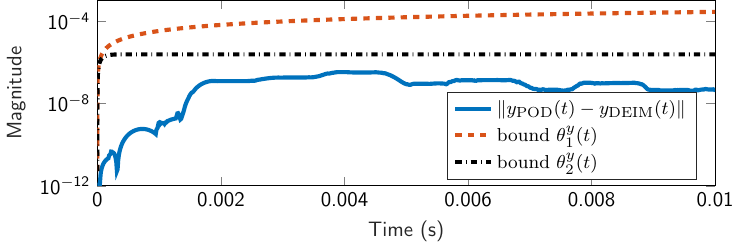} 
\caption{Absolute error in the output and error bounds}
\label{fig:NumExNonouterror}
\end{figure}

In Table~\ref{tab:DEIMerrorconst}, we collect the constants involved in the DEIM state error bound $\theta$ in~\eqref{eq:epsilon} required for the passivi\-ty enforcement of the POD-DEIM system \eqref{eq:POD-DEIMred}. Figure~\ref{fig:NumExNonstateerror} shows the absolute error ${\|\varepsilon(t)\|=\|x_{\POD}(t)-x_{\DEIM}(t)\|}$ and 
the state error bounds~$\theta_i(t)$ as in \eqref{eq:epsilon} computed with $\mu=\mu_i$, $i=1,2$, as in Remark~\ref{rem:logLip2D}. One can see that 
the error bound $\theta_1(t)$ overestimates the true error by about two orders of magnitude, while the error bound $\theta_2(t)$ is quite sharp.
In Figure~\ref{fig:NumExNonouterror}, we present the output error $\|y_{\POD}(t)-y_{\DEIM}(t)\|$ and the output error bounds $\theta^y_i(t)=\|C\|\theta_i(t)$ for $i=1,2$. 

\begin{figure}[tbp] %3
\centering
\hspace*{2mm}
\includegraphics{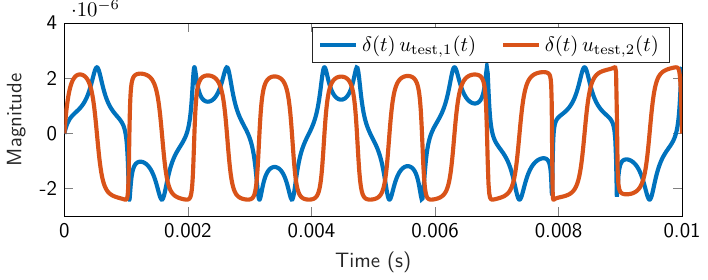} 
\caption{Components of the perturbation $\delta(t)u_{\rm test}(t)$} 
\label{fig:NumExNondelta}
\end{figure}

\begin{figure}[ht] %4
\centering
\includegraphics{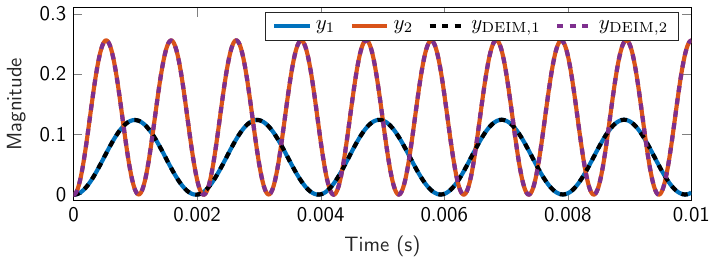} 
\caption{Output components of the original and POD-DEIM-reduced models}
\label{fig:NumExNonout}
\end{figure}

To enforce passivity, the output of the POD-DEIM model \eqref{eq:POD-DEIMred} with $u(t)=u_{\rm test}(t)$ as in \eqref{eq:NonlinExinput} is perturbed by $\delta(t) u_{\rm test}(t)$, where $\delta(t)$ is given in \eqref{eq:delta}.
Figure~\ref{fig:NumExNondelta} shows the components of the perturbation $\delta(t) u_{\rm test}(t)$. In Figure~\ref{fig:NumExNonout}, we present the output components $y_i$ and $y_{\delta, i}$, $i=1,2$, of the original and perturbed POD-DEIM systems. The relative errors $\Delta(y,y_{\DEIM})$ and $\Delta(y,y_{\delta})$ for the POD-DEIM and perturbed POD-DEIM  systems are given in Figure~\ref{fig:NumExNonerlerrordelta}. Here, the relative error is defined as 
\[
\Delta(y, z)=\sqrt{\sum_{i=1}^2
	\left(\frac{y_i(t)-z_i(t)}{\max\limits_{\tau\in [0,0.01]} |y_i(\tau)|}\right)^2}.
\]
We see that the error of the perturbed system \eqref{eq:delta-sys}
is only slightly larger than that of the POD-DEIM system \eqref{eq:POD-DEIMred}. 

\begin{figure}[ht] %5
\centering
\includegraphics{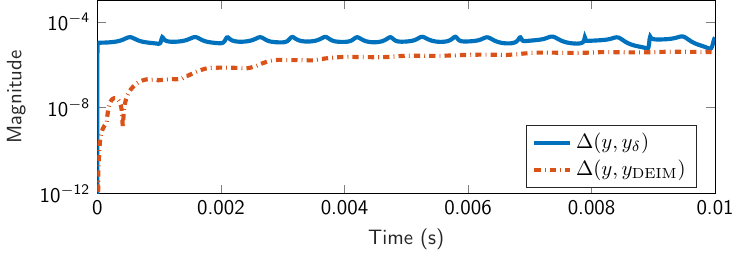} 
		\caption{Relative errors in the output for the POD-DEIM and perturbed POD-DEIM systems} 
\label{fig:NumExNonerlerrordelta}
\end{figure}

\section{Conclusion}
\label{sec:concl}

In this paper, we have studied MQS systems arising in the simulation of 
low-frequency electromagnetic devices coupled to electrical circuits. 
Passivity of such systems in the strong and weak formulations has been analyzed
by defining a~storage function which describes the magnetic energy of the system. 
A~FEM discretization of the MQS field problems on 3D domains leads to a~singular DAE system. 
We have investigated the structural properties of the resulting nonlinear DAE system and
presented a~new regularization approach based on projecting out singular state components. 
For this purpose, a~condensed form for the underlying system pencil has been derived which 
allows to decompose the semidiscretized MQS system into the regular and singular parts and 
to determine the subspaces corresponding to the infinite and zero generalized eigenvalues.
This makes it possible to transform the regularized system into the ODE form and to apply 
the POD-DEIM model reduction method. For the FEM model, its regularized formulation and the POD model, we have proved that passivity is preserved. Furthermore, for the POD-DEIM model, we have presented a~passivity enforcement method based on perturbation of the output which depends on the errors introduced by DEIM. Numerical experiments for a~model problem demonstrate the performance of the presented model reduction methods and the passivity enforcement technique.

\section*{Acknowledgments} We would like to thank Caren Tischendorf for providing the sparse DAE solver PyDAESI.

%===============================================================================

\end{document}